\title{Max-plus Linear Inverse Problems: \\ 2-norm regression and system identification of max-plus linear dynamical systems with Gaussian noise}
\author{James Hook}

\documentclass[12pt]{article}
\usepackage{eqnarray,amssymb}
\usepackage{amsmath,amsthm}
\usepackage{pgf}
\usepackage{tikz}
\usetikzlibrary{arrows,automata}
\usepackage[latin1]{inputenc}
\usepackage{graphicx}
\newtheorem{theorem}{Theorem}[section]
\newtheorem{lemma}[theorem]{Lemma}

\newtheorem{example}[theorem]{Example}
\newtheorem{corollary}[theorem]{Corollary}

\usepackage{float}
\usepackage{enumitem}
\usepackage{subfigure}
\usepackage{caption}
\usepackage{longtable}
\usepackage{stmaryrd}

\usepackage{bm}

\DeclareFontFamily{OT1}{pzc}{}
\DeclareFontShape{OT1}{pzc}{m}{it}{<-> s * [1.10] pzcmi7t}{}
\DeclareMathAlphabet{\mathpzc}{OT1}{pzc}{m}{it}

\usepackage{algorithm}
\usepackage{algpseudocode}

\DeclareFontFamily{OT1}{pzc}{}
\DeclareFontShape{OT1}{pzc}{m}{it}{<-> s * [1.10] pzcmi7t}{}
\DeclareMathAlphabet{\mathpzc}{OT1}{pzc}{m}{it}

\def\uc#1{\mathcal{#1}}

\def\Rmax{\mathbb{R}_{\max}}

\def\R{\mathbb{R}}

\def\col{\hbox{\normalfont col}}
\def\support{\hbox{\normalfont support}}

\def\pattern{\hbox{\normalfont pattern}}
\def\Cl{\hbox{\normalfont Cl}}

\def\relint{\hbox{\normalfont relint}}

\def\uc{\mathcal}

\usepackage{subfigure}
\usepackage{caption}

\newtheorem{problem}[theorem]{Problem}

\theoremstyle{definition}
\usepackage[hmargin=2cm,vmargin=2cm]{geometry}

\begin{document}
\maketitle
\begin{abstract}

In this paper we present new theory and algorithms for $2$-norm regression over the max-plus semiring. As an application we also show how max-plus $2$-norm regression can be used in system identification of max-plus linear dynamical systems with Gaussian noise. We also introduce and provide methods for solving a max-plus linear inverse problem with regularization, which can be used when the the original problem is not well posed.

\end{abstract}

\section{Introduction}

Max-plus algebra concerns the max-plus semiring $\Rmax=[\R\cup\{-\infty\},\oplus,\otimes]$, with
\begin{equation}
a\oplus b=\max\{a,b\}, \quad a\otimes b=a+b, \quad \hbox{for all $a,b\in\Rmax$}.
\end{equation}
A max-plus matrix is an array of elements from $\Rmax$ and max-plus matrix multiplication is defined in analogy to the classical (i.e. not max-plus) case. For $A\in\Rmax^{n\times d}$ and $B\in\Rmax^{d\times m}$ we have $A\otimes B\in\Rmax^{n\times m}$ with 
\begin{equation}
(A\otimes B)_{ij}=\max_{k=1}^{d}(a_{ik}+b_{kj}),
\end{equation}
for $i=1,\dots,n$, $j=1,\dots,m$. Max-plus algebra has found a wide range of applications in operations research, dynamical systems  and control, see for example \cite{butk10,heid2006,DESCHUTTER20011049} and the references therein. This paper concerns the following max-plus regression problem with $p=2$. To the best of our knowledge we are the first to study this problem explicitly.  
\begin{problem}\label{mppnreg}
For $A\in\Rmax^{n\times d}$,  $\bm{y}\in\Rmax^n$ and $p\geq 1$, we seek
\begin{equation}\label{red443}
\min_{\bm{x}\in\Rmax^d}\|A\otimes \bm{x}-\bm{y}\|_{p}.
\end{equation}
\end{problem}
In order for Problem~\ref{mppnreg} to be well defined we need to be able to measure the $p$-norm distance between arbitary vectors in $\Rmax^n$, which we do as follows. For $\bm{x}\in\Rmax^n$ define $\support(\bm{x})\subseteq\{1,\dots,n\}$, by $i\in\support(\bm{x})$ $\Leftrightarrow$ $\bm{x}_{i}>-\infty$. Then for $\bm{x},\bm{y}\in\Rmax^{n}$, we use the convention that
\begin{equation}\label{2_norm}
\|\bm{x}-\bm{y}\|_{p}=\left\{\begin{array}{cc}  \|\bm{x}_{\support(\bm{x})}-\bm{y}_{\support(\bm{y})}\|_{p}, & \hbox{if $\support(\bm{x})=\support(\bm{y})$,} \\ \infty, & \hbox{otherwise,}  \end{array}\right. 
\end{equation}
where $\bm{x}_{\support(\bm{x})}$ is the sub-vector of $\bm{x}$ formed from its finite entries.

\subsection{System identification}

Virtually every application of max-plus algebra in dynamical systems and control exploit its ability to model certain classically non-linear phenomena in a linear way, as illustrated in the following example.

\begin{example}Consider a distributed computing system in which $d$ processors update a computation in parallel. At each stage processor $i$ must wait until it has received input from its neighboring processors before beginning its next local computation. Then after completing its local computation it must broadcast some output back to its neighboring processors. Define the vectors of update times $\bm{t}(0),\dots,\bm{t}(N)\in\Rmax^d$, by $\bm{t}(n)_{i}=$ the time at which processor $i$ completes its $n$th local computation. These update times can be modeled by 
\begin{equation}\label{ffffffff}
\bm{t}(n+1)=M\otimes \bm{t}(n), \quad \hbox{for $n=0,\dots,N-1$,}
\end{equation}
where $M\in\Rmax^{d\times d}$ is the max-plus matrix given by
\begin{equation}
m_{ij}=\left\{\begin{array}{cc} a_{i}+c_{ij}, & \hbox{if $j\in J_{i}$}, \\ -\infty, & \hbox{otherwise}, \end{array}\right.
\end{equation}
where $a_{i}$ is the time taken for processor $i$'s local computation, $c_{ij}$ is the time taken for communication from processor $j$ and processor $i$ receives input from the processors $J_{i}\subset\{1,\dots,d\}$, for $i,j=1,\dots,d$. The update rule \eqref{ffffffff} constitutes a max-plus linear dynamical system. By studying the max-plus algebraic properties of the matrix $M$ we can now predict the behavior of the system, for example computing its leading eigenvalue to determine the average update rate of the computations iteration. 
\end{example}

Using the petri-net formalism, such max-plus linear  models can be derived for more complicated systems of interacting timed events  \cite[Chapter 7]{heid2006}. These linear models can be extended by introducing stochasticity, which in the above example could model random variability in the time taken for messages to pass through the computer network \cite[Chapter 11]{hook2013,heid2006}, by allowing the system to switch between one of several governing max-plus linear equations \cite{vandenBoom2012}, or by including a controller input \cite{DESCHUTTER20011049}. This approach has been used to model a wide variety of processes including the Dutch railway system \cite[Chapter 8]{heid2006}, mRNA translation \cite{BRACKLEY2012128} and the Transmission Control Protocol (TCP) \cite{Baccelli:2000:TML:347057.347548}. 

In this context \emph{forwards problems} arise by presupposing a dynamical systems model then asking questions about how its orbits must behave. Conversely an \emph{inverse problem} is to infer a dynamical systems model from an empirical time series recording. In the control theory literature this inverse problem is referred to as \emph{system identification}. For example in \cite{Farahani2014,1184996,7085024,7082376} the authors present methods for system identification of stochastic max-plus linear control systems. These methods, which can be applied to a very wide class of system, with non-Gaussian noise processes, work by formulating a non-linear programming problem for the unknown system parameters, which is then solved using one of several possible standard gradient based algorithm. However, the resulting problems are necessarily non-smooth and non-convex, which makes the optimization difficult. 

\subsection{This paper's contribution}

In this paper we introduce the max-plus $2$-norm regression problem and show how system identification of max-plus linear dynamical systems with Gaussian noise can be expressed in this way. Doing so enables us to apply existing theory from the max-plus linear algebra literature, along with some new theory developed in this paper, to better understand the geometric and combinatorial aspects of the system identification problem. In particular we are able to demonstrate the geometric cause of the non-convexity and prove that in general determining whether or not a point is a local minimum of the residual is an NP-hard problem. 

We provide two algorithms. One is an exponential cost method that is guaranteed to return the global minimum. The other is based on Newton's method and has cost $\Theta(nd)$ per iteration but cannot be guaranteed to converge even to a local minimum, although as demonstrated in the example problem we find that it tends to work well in practice. 

We also present a regularized version of the max-plus $2$-norm regression problem, which may be useful in data analysis applications. We also present a method to solve this regularized problem which we name Iteaitivley Reshifted Least Squares since it bears a striking similarity to the well known Iteratively Reweighed Least Squares method, which is used to solve classical $1$-norm regularized $2$-norm regression problems.  

The remainder of this paper is organized as follows. In Section 2 we introduce max-plus regression with a small example demonstrating the $p=2$ and $p=\infty$ cases. In Section 2.1 we make a detailed study of the $p=2$ case. In Subsection~\ref{cl} we develop some necessary machinery for working with max-plus sets and functions. In Subsection~\ref{fifo} we show how any max-plus $2$-norm regression problem can be reduced to a smaller problem with finite right hand side and at least one finite entry per row. In Subsections~\ref{pat1},~\ref{pat2},~\ref{pat3} and \ref{pat4} we develop theory for the combinatorial and geometric aspects of the max-plus $2$-norm regression problem before giving a small explicit example. In Section 2.2 we prove the NP-hardness result. In Sections 2.3 and 2.4 we introduce our two algorithms and in Section 3 we apply our new theory and algorithms to a systems identification example using both the original and regularized forms of the problem.

\section{Max-plus regression}\label{reg}
The column space of a max-plus matrix $A\in\Rmax^{n\times d}$ is simply the image of the matrix vector multiplication map
\begin{equation}
\col(A)=\{A\otimes \bm{x}~:~\bm{x}\in\Rmax^d\}.
\end{equation}
Just as in the classical case, the $p$-norm regression problem can be written as an optimization over the column space of the matrix
\begin{equation}\label{colview}
\min_{\bm{x}\in\Rmax^d}\|A\otimes \bm{x}-\bm{y}\|_{p}=\min_{\bm{z}\in\col(A)}\|\bm{z}-\bm{y}\|_{p}.
\end{equation}
Understanding the geometry of the column space is therefore key to understanding the regression problem. For more detail on max-plus linear spaces see \cite{firststeps} and the references therein.
\begin{example}\label{inf2egg}
Consider 
$$
A=\left[\begin{array}{cc} 0 & 0 \\ 1 & 0 \\ 0 & 1 \end{array}\right], \quad \bm{y}=\left[\begin{array}{cc} 1  \\ 1  \\ 1 \end{array}\right].
$$
The column space of $A$ is given by the union of two simplices  
$$
\col(A)=\{[x_{1},x_1+1,x_2+1]^{\top}~:~x_2 \leq x_1 \leq x_2+1\}\bigcup\{[x_2,x_1+1,x_2+1]^{\top}~:~x_1 \leq x_2 \leq x_1+1\}.
$$
Equivalently $\col(A)$ is a prism with an L-shaped cross section
$$
\col(A)=\{\bm{l}+[\alpha,\alpha,\alpha]^{\top}~:~\bm{l} \in L,~\alpha\in\Rmax\},
$$
where
$$
L=\{[0,t,1]^{\top}~:~t\in[0,1]\}\bigcup \{[0,1,t]^{\top}~:~t\in[0,1]\}.
$$
Now consider Problem~\ref{mppnreg} with $p=\infty$. Figure~\ref{pinfp2} (a) displays the column space of $A$ along with the target vector $\bm{y}$. We have also plotted the ball 
$$
B_{\infty}(\bm{y},1/2)=\{\bm{y'}\in\Rmax^3~:~\|\bm{y}'-\bm{y}\|_{\infty}=1/2\},
$$
which is the smallest such ball that intersects $\col(A)$. Therefore the minimum value of the residual is $1/2$ and the closest points in the column space are given by the L-shaped 
set 
$$
\arg\min_{\bm{z}\in\col(A)}\|\bm{z}-\bm{y}\|_{\infty}=\{[1/2,3/2,t]^{\top}~:~t\in[1/2,3/2]\}\bigcup \{[1/2,t,3/2]^{\top}~:~t\in[1/2,3/2]\}.
$$
Next consider Problem~\ref{mppnreg} with $p=2$. Figure~\ref{pinfp2} (b) displays the column space of $A$ along with the target vector $\bm{y}$. We have also plotted the ball 
$$
B_{2}(\bm{y},1/\sqrt{2})=\{\bm{y'}\in\Rmax^3~:~\|\bm{y}'-\bm{y}\|_{2}=1/\sqrt{2}\},
$$
which is the smallest such ball that intersects $\col(A)$. Therefore the minimum value of the residual is $1/\sqrt{2}$ and the closest points in the column space are given by
$$
\arg\min_{\bm{z}\in\col(A)}\|\bm{z}-\bm{y}\|_{\infty}=\{[1/2,3/2,1]^{\top},[1/2,1,3/2]^{\top}\}.
$$

\end{example}

\begin{figure}[t]
\begin{center}
\subfigure[$p=\infty.$]{\includegraphics[scale=0.225]{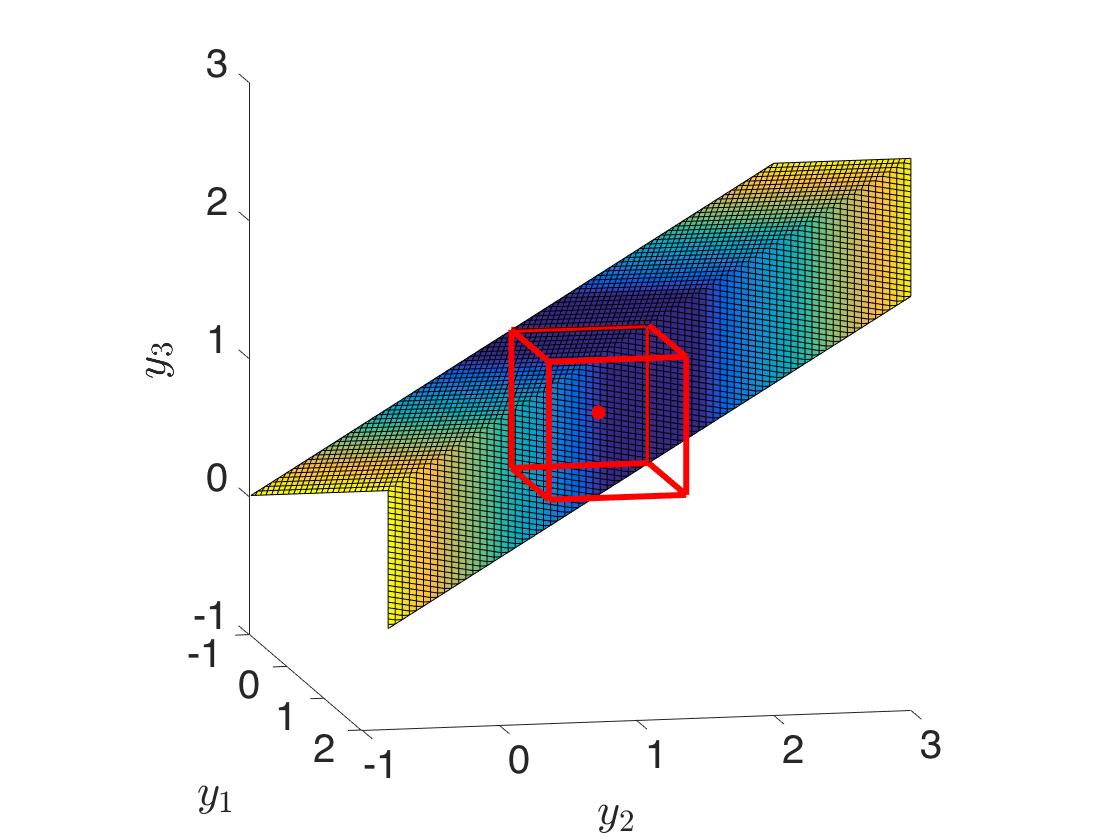}}\subfigure[$p=2.$]{\includegraphics[scale=0.225]{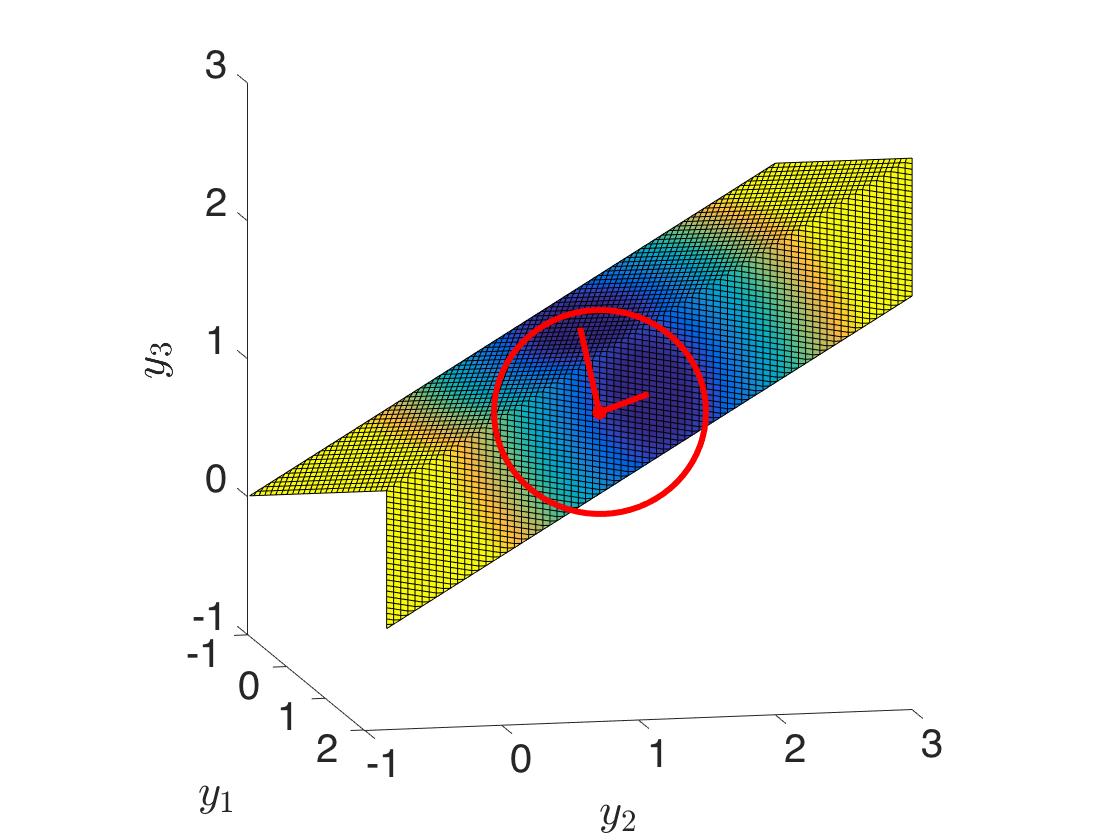}} 
\caption{Column space view of max-plus regression.}\label{pinfp2}
\end{center}
\end{figure}

We saw in Example~\ref{inf2egg} that the max-plus $\infty$-norm regression problem could support multiple optimal solutions comprising a non-convex set. However it can be shown that the max-plus $\infty$-norm regression problem is convex with respect to max-plus algebra \cite{73810} and that we can fairly easily compute an optimal solution for it \cite[Section 3.5]{butk10}. We also saw that the max-plus $2$-norm regression problem could support multiple isolated local minima and was therefore non-convex. However, because these local minima do not form a single, path connected set, this example also shows that the problem is max-plus non-convex. 

\subsection{2-norm regression}\label{2normo}

Since the max-plus $2$-norm regression problem is non-convex in both classical and max-plus senses it is more difficult to solve than the $\infty$-norm problem. However it may be more useful in application as $2$-norm regression corresponds to computing a maximum likelihood estimate for a Bayesian inverse problem with Gaussian noise, as illustrated by the example in Section 3. In the remainder of this section we make a detailed study of the max-plus $2$-norm regression problem. 
%
%

\subsection{Max-plus closure}\label{cl}

Our approach to max-plus regression will be to develop classical mathematical description of all of the different max-plus objects that play an important role in the problem. In order to do this we need to be able to take closures of sets in $\Rmax^n$ and extend or continue the definition of certain functions from $\R^n$ to $\Rmax^n$. 

We say that the sequence $(a_{t}\in \R)_{t=1}^{\infty}$ converges to $\hat{a}\in\Rmax$, if either $\hat{a}\in\R$ and the sequence converges to $\hat{a}$ in the usual sense, or $\hat{a}=-\infty$ and for all $m\in\mathbb{R}$, there exists $t_{0}$ such that $a_{t}\leq m$, for all $t \geq t_{0}$. This definition of convergence is extended componentwise to vectors in the obvious way. The max-plus closure of $X\subseteq \R^{n}$, denoted $\Cl(X)$, consist of all such limit points of sequences in $X$. For example $\Cl(\{0\})=\{0\}$, $\Cl(\{\bm{x}\in\mathbb{R}^2~:~\bm{x}_1>\bm{x}_2\})=\{\bm{x}\in\Rmax^2~:~\bm{x}_{1}\geq \bm{x}_{2}\}$ and $\Cl(\R^n)=\Rmax^n$.

A function $f:\R^n\mapsto \R$ is \emph{monotonic} if $f(\bm{x})\leq f(\bm{y})$, for all $\bm{x},\bm{y}\in\R^n$ such that $\bm{x}\leq \bm{y}$, where $\leq$ is the standard partial order on $\R^n$.  A sequence $(\bm{a}_{t}\in\R^n)_{t=1}^{\infty}$ is \emph{monotonically non-increasing} if $\bm{a}_{t+1}\leq \bm{a}_{t}$, for all $t=1,2,\dots$. Given a uniformly continuous monotonic function $f:\R^n \mapsto \R$, we \emph{continue} its definition to $f:\Rmax^n \mapsto \Rmax$ using the rule that
\begin{equation}\label{mpexp}
f(\hat{\bm{a}})=\lim_{t\rightarrow\infty}f(\bm{a}_{t}),
\end{equation}
whenever the monotonically non-increasing sequence $(\bm{a}_{t}\in\R^n)_{t=1}^{\infty}$ converges to $\hat{\bm{a}}\in\Rmax^n$. 
\begin{theorem}Let $f:\R^n\mapsto \R$ be a uniformly continuous monotonic function. Then the continuation $f:\Rmax^n \mapsto \Rmax$ is well defined. 
\end{theorem}
\begin{proof}
Suppose that $(\bm{a}_{t}\in\R^n)_{t=1}^{\infty}$ is a monotonically non-increasing sequences that converges to $\hat{\bm{a}}\in\Rmax^n$. Consider the sequences $(\bm{b}_{t})_{t=1}^{\infty}$ and $(\bm{c}_{t})_{t=1}^{\infty}$ defined by
$$
(\bm{b}_{t})_{i}=\left\{\begin{array}{cc} (\hat{\bm{a}})_{i} & \hbox{if $i\in\support(\hat{\bm{a}})$}, \\ -t & \hbox{otherwise.} \end{array}\right. \quad (\bm{c}_{t})_{i}=\left\{\begin{array}{cc} (\hat{\bm{a}})_{i} & \hbox{if $i\in\support(\hat{\bm{a}})$}, \\ (\bm{a}_{t})_{i} & \hbox{otherwise.} \end{array}\right.
$$
Clearly $(\bm{b}_{t})_{t=1}^{\infty}$ also converges to $\hat{\bm{a}}$ and $\lim_{t\rightarrow\infty}f(\bm{b}_{t})$ exists in the sense of convergence introduced above because it is the limit of a non-increasing sequence. We now show that $\big(f(\bm{a}_{t})\big)_{t=1}^{\infty}$ has the same limit. 

By uniform continuity, for all $\epsilon>0$ there exists $\delta>0$, such that $|f(\bm{x})-f(\bm{y})|\leq \epsilon$, whenever $\|\bm{x}-\bm{y}\|_{2}\leq \delta$. By the definition of convergence given above, for all such $\epsilon,\delta>0$ and $m \in\mathbb{N}$ there exists $T(\epsilon,m)$ such that
$$
(\bm{a}_{t})_{i}\leq \left\{ \begin{array}{cc} (\hat{\bm{a}})_{i}+\frac{\delta}{\sqrt{n}} & \hbox{if $i\in\support(\hat{\bm{a}})$}, \\ -m & \hbox{otherwise},\end{array}\right.
$$
for all $i=1,\dots,n$ and all $t\geq T(\epsilon,m)$. Note that $\|\bm{c}_{t}-\bm{a}_{t}\|_{2}\leq \delta$, so that $|f(\bm{c}_{t})-f(\bm{a}_{t})|\leq \epsilon$, and $\bm{c}_{t}\leq \bm{b}_{m}$, for all $t\geq T(\epsilon,m)$. Therefore
$$
f(\bm{a}_{t})\leq f(\bm{b}_{m})+\epsilon
$$
and taking the limit $\epsilon\rightarrow 0$ and the limit $m\rightarrow\infty$ we obtain 
\begin{equation}\label{oneway}
\lim_{t\rightarrow\infty}f(\bm{a}_{t})\leq \lim_{t\rightarrow\infty}f(\bm{b}_{t}).
\end{equation}
Next consider the sequence $\big(B(t)\in\mathbb{N}\big)_{t=1}^{\infty}$, defined by
$$
B(t)=\max\Big\{0,-\min_{i=1}^{n}\lfloor (\bm{a}_{t})_{i}\rfloor\Big\},
$$
where $\lfloor \cdot \rfloor$ denotes the integer floor. Then the subsequence $\big(\bm{b}_{B(t)}\big)_{t=1}^{\infty}$ satisfies $\bm{b}_{B(t)}\leq \bm{a}_{t}$, for all $t=1,2,\dots$. Therefore
$$
\lim_{t\rightarrow\infty}f(\bm{a}_{t})\geq \lim_{t\rightarrow\infty}f(\bm{b}_{B(t)})=\lim_{t\rightarrow\infty}f(\bm{b}_{t}),
$$
which together with \eqref{oneway} completes the proof. 
\end{proof}

\subsection{Reduction to finite form}\label{fifo}

In this section we show that any instance of Problem~\ref{mppnreg} can either, be reduced to a problem with finite right hand side and at least one finite entry per row, or does not admit any solution with finite residual. 

We say that $A\in\Rmax^{n\times d}$ and $\bm{y}\in\Rmax^n$ are a \emph{finite form}, if $\support(\bm{y})=\{1,2,\dots,n\}$ and $\max_{j=1}^{d}a_{ij}>-\infty$, for all $i=1,\dots,n$. Also define the residual $R:\Rmax^d\mapsto \R\cup\{+\infty\}$, by
$$
R(\bm{x})=\|A\otimes\bm{x}-\bm{y}\|_{2}^2/2.
$$
\begin{theorem}\label{ffthm}
If $A\in\Rmax^{n\times d}$ and $\bm{y}\in\Rmax^n$ are a finte form, then $R(\bm{x})<\infty$, for all $\bm{x}\in\R^d$.
\end{theorem}
\begin{proof}
From definition \eqref{2_norm} and the fact that $\support(\bm{y})=\{1,2,\dots,n\}$, we have that $R(\bm{x})<\infty$ whenever 
$$
\support(A\otimes \bm{x})=\{1,2,\dots,n\}.
$$
Since each row of $A$ contains at least one finite entry, there exists $c:\{1,\dots,n\}\mapsto\{1,\dots,d\}$, such that $a_{ic(i)}>-\infty$, for all $i=1,\dots,n$. Then for any $\bm{x}\in\R^d$, we have
$$
(A\otimes \bm{x})_{i}=\max_{j=1}^{d}a_{ij}+\bm{x}_{j}\geq a_{ic(i)}+\bm{x}_{c(i)}>-\infty,
$$
for all $i=1,\dots,n$. So that $R(\bm{x})<\infty$. 
\end{proof}

We say that $j\in\{1,\dots,d\}$ is \emph{col-admissible}, if 
$$
i\not\in\support(\bm{y})~\Rightarrow~a_{ij}=-\infty.
$$
We say that $i\in\support(\bm{y})$ is \emph{row-admissible}, if
$$
\exists~\hbox{col-admissible }j~:~a_{ij}>-\infty.
$$
Let $\uc{C}$ denote the set of all col-admissible indices and $\uc{R}$ denote the set of all row-admissible indices.

%
%
\begin{lemma}\label{Fflem}
Let  $A\in\Rmax^{n\times d}$ and $\bm{y}\in\Rmax^n$. Then 
\begin{equation}\label{conditionrc}
\support(\bm{x})\subseteq \uc{C}
\end{equation}
is a necessary condition for $R(\bm{x})<\infty$.
\end{lemma}
\begin{proof}
Suppose there exists $j\in\{1,\dots,d\}$ such that $j\not\in \uc{C}$ but $\bm{x}_j>-\infty$. By the definition of col-admissibility there must exist $i\in\{1,\dots,n\}$, such that $\bm{y}_{i}=-\infty$ and $a_{ij}>-\infty$. Therefore
$$
(A\otimes\bm{x})_{i}=\max_{k=1}^{d}a_{ik}+\bm{x}_{k}\geq a_{ij}+\bm{x}_{j}>-\infty
$$
and since $\bm{y}_{i}=-\infty$ we have $R(\bm{x})=\infty$.
\end{proof}

\begin{theorem}\label{ffreduction}
Let  $A\in\Rmax^{n\times d}$ and $\bm{y}\in\Rmax^n$.
\begin{enumerate}
 \item If $\uc{R}\neq \support(\bm{y})$, then 
$R(\bm{x})=\infty$, for all $\bm{x}\in\Rmax^d$. 

\item If  $\uc{R}= \support(\bm{y})$, then whenever $\bm{x}\in\R^d$ satisfies \eqref{conditionrc}, we have
$$
R(\bm{x})=\|A_{\uc{R}\uc{C}}\otimes \bm{x}_{\uc{C}}-\bm{y}_{\uc{R}}\|_{2}.
$$
where $A_{\uc{C},\uc{R}}$ denotes the sub-matrix of $A$ formed from the rows in $\uc{R}$ and columns in $\uc{C}$, likewise for the sub-vector $\bm{y}_{\uc{R}}$. Moreover $A_{\uc{C},\uc{R}},\bm{y}_{\uc{R}}$ are a finite form.
\end{enumerate}
\end{theorem}
\begin{proof}
For the first part first suppose that \eqref{conditionrc} is not satisfied then $R(\bm{x})=\infty$. Now suppose that  \eqref{conditionrc} is satisfied. From $R\neq\support(\bm{y})$, we have that there exists $i$ with $\bm{y}_{i}>-\infty$ and $i\not\in\uc{R}$. From the definition of row-admissibility, we have $a_{ij}=-\infty$, for all $j\in\uc{C}$. So that $(A\otimes \bm{x})_{i}=-\infty$, unless $\bm{x}_{k}>-\infty$, for some $k\not\in\uc{C}$, and this would violate \eqref{conditionrc}. Therefore $(A\otimes \bm{x})_{i}=-\infty$ and $\bm{y}_{i}>-\infty$, so $R(\bm{x})=\infty$. 

For the second part, from \eqref{conditionrc} and the definition of col-admissibility we have that
$$
\support(A\otimes \bm{x})\subseteq \support(\bm{y}).
$$
So from the assumption $\uc{R}=\support(\bm{y})$ and \eqref{2_norm} we have
$$
R(\bm{x})=\|A\otimes \bm{x}-\bm{y}\|_{2}=\|(A\otimes \bm{x})_{\uc{R}}-\bm{y}_{\uc{R}}\|_{2},
$$
and applying  \eqref{conditionrc} a second time we have
$$
R(\bm{x})=\|A_{\uc{R}\uc{C}}\otimes \bm{x}_{\uc{C}}-\bm{y}_{\uc{R}}\|_{2}.
$$
For the last part. From the assumption $\uc{R}= \support(\bm{y})$, we have that every entry in $\bm{y}_{\uc{R}}$ is finite. Now consider an arbitrary row of $A_{\uc{R}\uc{C}}$. This will be the $i$th row of $A$, for some $i\in\uc{R}$. From the definition of row-admissibility we have that there exists $j\in\uc{C}$, such that $a_{ij}>-\infty$, and this will correspond to a finite entry in the chosen row of $A_{\uc{R}\uc{C}}$. Therefore $A_{\uc{C},\uc{R}},\bm{y}_{\uc{R}}$ are a finite form.

\end{proof}

Theorem~\ref{ffreduction} tells us that given any max-plus $2$-norm regression problem, either no solution with finite residual exists, or the problem can be reduced to a smaller problem in finite form, for which Theorem~\ref{ffthm} guarantees that finite solutions with finite residuals exist. 

\subsubsection{Patterns of support and their domains}\label{pat1}

For $A\in\Rmax^{n\times d}$ and $\bm{x}\in\Rmax^d$ define the \emph{pattern of support} $\pattern(\bm{x})=(P_{1},\dots,P_{n}) \in \uc{P}\big(\{1,\dots,d\}\big)^n$, by
\begin{equation}
 j\in P_{i} \quad \Leftrightarrow \quad  a_{ij}+\bm{x}_{j}=(A\otimes \bm{x})_{j},
\end{equation}
for $i=1,\dots,n$. Define the \emph{domain} of a pattern $P\in \uc{P}\big(\{1,\dots,d\}\big)^n$, by 
\begin{equation}
X(P)=\{\bm{x}\in\Rmax^d~:~\pattern(\bm{x})=P\}.
\end{equation}
We say that a pattern $P$ is \emph{feasible}, if $X(P)\cap \R^d \neq \emptyset$. For a pattern $P$ define the binary relation $\tilde{\bowtie}_P$ on $\{1,\dots,d\}$, by $j\tilde{\bowtie}_P k$, if and only $j,k\in P_{i}$, for some $i=1,\dots,n$ and define $\bowtie_{P}$ to be the transitive closure of $\tilde{\bowtie}_P$. Let $|\bowtie_{P}|$ be the number of equivalence classes of $\bowtie_{P}$. If $P$ is a feasible pattern then $X(P)$ is a set of dimension $|\bowtie_{P}|$.  From \cite[Cor.~25]{73810} we have that
\begin{equation}\label{numer}
|\{\hbox{feasible $P$} ~:~ |\bowtie_{P}|=k\}|\leq \frac{(n+d-k-1)!}{(n-k)!\cdot (d-k)! \cdot (k-1)!},
\end{equation}
for $k=1,\dots,d$, with equality for all $k$ in the generic case of a matrix $A$ with rows/cols in general position. 

Define the ordering $\preceq$ on $\uc{P}\big(\{1,\dots,d\}\big)^n$, by $P\preceq P'$, if and only if $P_{i}\subseteq P_{i}'$, for all $i=1,\dots,n$, with a strict inequality if at least one inclusion is a strict inclusion. Then the boundary of the domain $X(P)$ is given by $\cup_{\{P'~:~P\prec P' \}}X(P')$ and the closure by $\Cl\big(X(P)\big)=\cup_{\{P'~:~P\preceq P' \}}X(P')$.


Also define the \emph{feasibility matrix} by $F_{P} \in\Rmax^{d\times d}$, by
\begin{equation}\label{feasibilitymatrixdef}
f_{jk}=\left\{\begin{array}{cc} 0, & \hbox{for $j=k$,} \\ \max\big\{-\infty,\max \{a_{ik}-a_{ij}~:~ j\in P_{i}~:~i=1,\dots,n\}\big\},   & \hbox{otherwise.}\end{array}\right.
\end{equation}

We will need to quickly review some related results to support the following Theorem. For a max-plus matrix $B\in \Rmax^{d\times d}$, the \emph{maximum cycle mean} of $B$ is defined by
\begin{equation}
\lambda(B)=\max_{\zeta}\frac{W(\zeta)}{L(\zeta)},
\end{equation}
where the maximum is taken over cycles $\zeta=\big(\zeta(1)\mapsto \dots \mapsto \zeta(k)\mapsto \zeta(1) \big)\subset \{1,\dots,d\}$. The \emph{weight} of a cycle is the sum of its edge weights $W(\zeta)=b_{\zeta(1)\zeta(2)}+\dots+b_{\zeta(k-1)\zeta(k)}+b_{\zeta(k)\zeta(1)}$ and the \emph{length} of a cycle is its total number of edges $L(\zeta)=k+1$. The \emph{Klene star} of $B\in \Rmax^d$ is defined by 
\begin{equation}
B^{\star}=\lim_{t \rightarrow \infty}\big(I\oplus B \oplus B^{\otimes 2}\oplus \dots \oplus B^{\otimes t}\big),
\end{equation}
where $I\in\Rmax^d$ is the max-plus identity matrix, with zeros on the diagonal and minus infinities off of the diagonal. From \cite[Prop.~1.6.10 and Thm.~1.6.18]{butk10} we have that if $\lambda(B)\leq0$, then $B^{\star}$ exists and $\{x\in\Rmax^d~:~ B\otimes x=x\} =\col(B^{\star})$ and that if $\lambda(B)>0$, then $B^{\star}$ does not exist and $\{x\in\Rmax^d~:~ B\otimes x=x\}\cap\R^d=\emptyset$. 

For $B\in\Rmax^{d\times d}$ define $\overline{B}\in\Rmax^{d}$ to be the arithmetic mean of the columns of $B$. It follows from \cite[Thm.~3.3]{ssb09} that, provided $\lambda(B)\leq 0$, we have $\overline{(B^{\star})}\in\relint\big(\col(B^{\star})\big).$

\begin{theorem}\label{feasibilitythm}
For $A\in\Rmax^{n\times d}$ and $P\in \uc{P}\big(\{1,\dots,d\}\big)^n$ we have
$$
\Cl\big(X(P)\big)=\{\bm{x}\in\Rmax^d~:~ F_{P}\otimes \bm{x}=\bm{x}\}.
$$
Moreover, the pattern $P$ is feasible, if and only if $\lambda(F_{P})=0$ and in the case where $P$ is feasible, we have
$$
X(P)=\relint \big( \col (F_{P}^{\star})\big)
$$
and $\overline{(F_{P}^{\star})}\in X(P)$.
\end{theorem}
\begin{proof}
First note that $\bm{x}\in \Cl\big(X(P)\big)$, if and only if $(A\otimes \bm{x})_{i}=a_{ij}+\bm{x}_{j}$, for all $j\in P_{i}$, for all $i=1,\dots,n$, which is equivalent to $\max_{k=1}^{d}(a_{ik}+\bm{x}_{k})\leq a_{ij}+\bm{x}_{j}$, for all $j\in P_{i}$, for all $i=1,\dots,n$, which is equivalent to $F_{P}\otimes \bm{x}\leq \bm{x}$ and since $F_{P}$ has zeros on its diagonal this is equivalent to $F_{P}\otimes \bm{x}=\bm{x}$. Next from \cite[Prop.~1.6.10 and Thm.~1.6.18]{butk10}, we have that $\Cl\big(X(P)\big)\cap \R^d$ is non-empty, if and only if $\lambda(F_{P})\leq 0$ and since $F_{P}$ has zeros on the diagonal this is equivalent to the condition $\lambda(F_{P})=0$. In the case that $\lambda(F_{P})=0$ we also have $ \Cl\big(X(P)\big)=\col(F_{P}^{\star})$. Then note that
$$
X(P)=\cup_{\{P'~:~P\preceq P'\}}X(P')\big\backslash \cup_{\{P'~:~P\prec P'\}}X(P')
$$
is equal to $\col(F_{P}^{\star})$ minus its boundary, which is precisely $\relint \big( \col (F_{P}^{\star})\big)$. The final result follows immediately from  \cite[Thm.~3.3]{ssb09}. 
\end{proof}

If there are $m$ equivalence classes in $\bowtie_{P}$ then label them arbitrarily with $\{1,\dots,m\}$ and define $c:\{1,\dots,d\}\mapsto\{1,\dots,m\}$, such that $c(j)=k$, if and only if $j$ is in the $k$th equivalence class. Let $\#_{k}$ denote the number of elements in the $k$th equivalence class and define $C\in\R ^{d\times m}$, by 
\begin{equation}\label{cmatrix}
c_{jc(j)}=\frac{1}{\sqrt{\#_{c(j)}}}, 
\end{equation}
for $j=1,\dots,d$ and all other entries equal to zero. Then for any $\bm{x}_{P}\in \Cl\big(X(P)\big)\cap \R^d$, we have that
\begin{equation}\label{exdom}
\uc{A}\big(X(P)\cap \R^d\big)=\{C\bm{h}+\bm{x}_{P}~:~\bm{h}\in\R^m\}
\end{equation}
is the smallest affine subspace of $\Rmax^d$ containing $X(P)\cap \R^d$. We could choose $\bm{x}_{P}=\overline{(F^{\star}_{P})}$ but also need to consider the case where $\bm{x}_{P}$ represents the current state of one of the algorithms that we detail later. We call $\Cl \Big(\uc{A}\big(X(P)\cap \R^d \big)\Big)$ the \emph{extended domain} of $P$. 

\subsubsection{Local maps and images}\label{pat2}

Now define the \emph{subpattern} $p$ of $P$, by $p_{i}=\{\min(P_{i})\}$, for $i=1,\dots,n$. Since each element of a subpattern contains only a single element we can treat subpatterns like vectors and will write $p_{i}$ to mean the unique element of the set $p_{i}$. Also define $L\in\{0,1\}^{n\times d}$, by $l_{ip(i)}=1$ and all other entries equal to zero. Then define the \emph{local mapping} $A_{P}:\Rmax^{d} \mapsto \Rmax^n$, for real $\bm{x}\in\R^d$ by
\begin{equation}\label{localmapdef}
A_{P}(\bm{x})=L\bm{x}+\bm{y}_{P},
\end{equation}
where $(\bm{y}_{P})_{i}=a_{ip(i)}$, for $i=1,\dots,n$ and continue $A_{P}(\cdot)$ to $\Rmax^n$ via \eqref{mpexp}. It follows that $\support(\bm{y}_{P})=\{1,\dots,n\}$, whenever $A,\bm{y}$ are a finite form and $P$ is a feasible pattern. Note that $A_{P}(\bm{x})=A\otimes\bm{x}$, for all $\bm{x}\in \Cl\big(X(P)\big)$.

Define the \emph{image} $Y(P)=A_{P}\big(X(P)\big)$. Then we have 
\begin{equation}
\col(A)=\bigcup_{P}Y(P),
\end{equation}
where the union is taken over all feasible patterns.  We have that 
\begin{equation}
\uc{A}\big(Y(P)\cap \R^n \big)=\{LC\bm{h}+L\bm{x}_{P}+\bm{y}_{P}:~\bm{h}\in\R^m\},
\end{equation}
is the smallest affine subspace of $\R^n$ containing $Y(P)\cap \R^n$. We call  $\Cl\Big(\uc{A}\big(Y(P)\cap \R^n\big)\Big)$ the \emph{extended image} of $P$.

\subsubsection{Normal projections}\label{pat3}

For a feasible pattern $P$ define the \emph{normal projection} map $\Phi(P,\cdot):\Rmax^n\mapsto\Rmax^n$, by
\begin{equation}
\Phi(P,\bm{y})=\arg\min\{\|\bm{y}-\bm{y}'\|_{2}~:~\bm{y}'\in\Cl\Big(\uc{A}\big(Y(P)\cap \R^n\big)\Big)\}.
\end{equation}
Then for real $\bm{y}\in \R^n$ we have 
\begin{equation}
\Phi(P,\bm{y})=LC\bm{h}^{\ast}+L\bm{x}_{P}+\bm{y}_{P},
\end{equation} 
where the normal equations \cite{lsbook} give
\begin{equation}
\bm{h}^{\ast}=\big((LC)^{\top}LC\big)^{\dagger}(LC)^{\top}(\bm{y}-L\bm{x}_{P}-\bm{y}_{P}).
\end{equation}
Note that $LC\in\R^{n\times m}$ with 
\begin{equation}
(LC)_{i~c \big( p(i)\big)}=\frac{1}{\sqrt{\#_{c\big(p(i)\big)}}}
\end{equation}
and all other entries equal to zero. Hence we have that
\begin{equation}\label{iver}
\bm{h}^{\ast}_{k}=\sqrt{{\#_{k}}}~\times ~\overline{\{(\bm{y}-L\bm{x}_{P}-\bm{y}_{P})_{i} ~:~c \big( p(i)\big)=k\}},
\end{equation}
for $k=1,\dots,m$ such that $\{i~:~c \big( p(i)\big)=k\}\neq\emptyset$, and $\bm{h}^{\ast}_{k}=0$ otherwise, and where the overline in \eqref{iver} indicates taking the mean. Define the equivalence relation $\hat{\bowtie}_{P}$ on $\{1,\dots,n\}$, by $i\hat{\bowtie}_{P}j$, if and only if $p(i)\bowtie_{P} p(j)$. Then we have
\begin{equation}
\Phi(P,\bm{y})_{i}=\overline{\{(\bm{y}-L\bm{x}_{P}-\bm{y}_{P})_{j}~:~ i\hat{\bowtie}_{P} j\}}+(L\bm{x}_{P}+\bm{y}_{P})_{i},
\end{equation}
for $i=1,\dots,n$. We continue $\Phi(P,\cdot)$ to $\Rmax^n$ via \eqref{mpexp}. We say that $\Phi(P,y)$ is \emph{admissible}, if $\Phi(P,\bm{y})\in \Cl\big(Y(P)\big)$.

Also define 
\begin{equation}
A^{-1}_{P}\big(\Phi(P,\bm{y})\big) =\{\bm{x}\in\Cl\Big(\uc{A}\big(X(P)\cap \R^d \big)\Big)~:~A_{P}(\bm{x})=\Phi(P,\bm{y})\} 
\end{equation}
where for  real $y\in\R^n$, we have
\begin{equation}
A^{-1}_{P}\big(\Phi(P,\bm{y})\big) =\Cl\big(C\bm{h}^{\ast}+C\ker(LC)+\bm{x}_{P}\big).
\end{equation}
where 
\begin{equation}
(C\bm{h}^{\ast})_{j}=\overline{\{(\bm{y}-L\bm{x}_{P}-\bm{y}_{P})_{i}~:~c \big( p(i)\big)=c(j)\}},
\end{equation}
for $j=1,\dots,d$ and
\begin{equation}
C~\hbox{ker}(LC)=\hbox{span}\{\underline{e}_{j}\in\Rmax^d~:~ j\in \{1,\dots,d\}/\support(P) \},
\end{equation}
where $\support(P)\subset\{1,\dots,d\}$ is the \emph{support} of the pattern $P$, defined by $\support(P)=\cup_{i=1}^{n}P_{i}$.

Define the \emph{closest local minimum} map $\Psi(P,\bm{y},\cdot):\Rmax^d\mapsto\Rmax^d$, for real $\bm{y}\in \R^n$ by 
\begin{equation}
\Psi(P,\bm{y},\bm{x})=\arg\min\{\|\bm{x}-\bm{x}'\|_{2}~:~\bm{x}'\in A^{-1}_{P}\big(\Phi(P,\bm{y})\big)\}.
\end{equation}
Then for real $\bm{x}\in\R^d$, we have 
\begin{equation}\label{ggg}
\Psi(P,\bm{y},\bm{x})_{j}=\left\{\begin{array}{cc} (C\bm{h}^{\ast})_{j} &\hbox{ for $j\in\support(P)$}, \\ {x}_{j} & \hbox{otherwise,}\end{array}\right.
\end{equation}
for $j=1,\dots,d$. For all $\bm{y}\in\R^n$ we continue $\Psi(P,\bm{y},\cdot)$ to $\Rmax^d$ via \eqref{mpexp}, then for all $\bm{x}\in\Rmax^d$ we continue $\Psi(P,\cdot,\bm{x})$ to $\Rmax^n$ via \eqref{mpexp}.

\begin{theorem}\label{admissthm}Let $A\in\Rmax^{n\times d},\bm{y}\in\Rmax^n$ be a finite form and let $P\in \uc{P}\big(\{1,\dots,d\}\big)^n$ be a  feasible pattern. Then the normal projection $\Phi(P,\bm{y})$ is admissible, if and only if $\Psi(P,\bm{y},\underline{-\infty})\in \Cl\big(X(P)\big)$, where $\underline{-\infty}\in\Rmax^d$ is a vector with all entries equal to $-\infty$.
\end{theorem}
\begin{proof}
First suppose that $\Psi(P,\bm{y},\underline{-\infty})\in\Cl\big(X(P)\big)$. By definition we have that $A\otimes \Psi(P,\bm{y},\underline{-\infty})=\Phi(P,\bm{y})$ and from the supposition we have that $A\otimes \Psi(P,\bm{y},\underline{-\infty})\in \Cl\big(Y(P)\big)$. Therefore $\Phi(P,\bm{y})\in \Cl\big(Y(P)\big)$, or in words, $\Phi(P,\bm{y})$ is admissible.

Conversely suppose that $\Phi(P,\bm{y})$ is admissible. Note that from Theorem~\ref{feasibilitythm} we have that $\Psi(P,\bm{y},\underline{-\infty})\in X(P)$, if and only if
\begin{equation}\label{thmcondition}
F_{P}\otimes \Psi(P,\bm{y},\underline{-\infty})= \Psi(P,\bm{y},\underline{-\infty}).
\end{equation}
From the supposition we have that there exists $\hat{\bm{x}}\in \Cl\big(X(P)\big)$, such that $A\otimes \hat{\bm{x}}=\Phi(P,\bm{y})$. Next note that since the feasibility matrix has zeros on its diagonal it suffices to show that 
$$
F_{P}\otimes \Psi(P,\bm{y},\underline{-\infty})\leq \Psi(P,\bm{y},\underline{-\infty}).
$$
Using the fact that $\Psi(P,\bm{y},\underline{-\infty})$ is the infimum of $A^{-1}_{P}\big(\Phi(P,\bm{y})\big)$ with respect to the standard partial order, that $(F_{P}\otimes \bm{x})_{k}$ is a monotonically non-decreasing function of $\bm{x}$, for all $k\in\support(P)$, and \eqref{ggg}, we have
$$
\big(F_{P}\otimes \Psi(P,\bm{y},\underline{-\infty})\big)_{k} \leq \big(F_{P}\otimes \hat{\bm{x}}\big)_{k}= \hat{\bm{x}}_{k}= (C\bm{h}^{\ast})_{j}=\Psi(P,\bm{y},\underline{-\infty})_{k},
$$
for $k\in\support(P)$. Funally note from \eqref{ggg}, that for $j\in\{1,\dots,d\}\backslash\support(P)$, we have 
$$
\big(F_{P}\otimes \Psi(P,\bm{y},\underline{-\infty})\big)_{j}=\Psi(P,\bm{y},\underline{-\infty})_{j}.
$$
Therefore \eqref{thmcondition} holds and $\Psi(P,\bm{y},\underline{-\infty})\in \Cl\big(X(P)\big)$.

\end{proof}

\subsubsection{Residual surface}\label{pat4}

For $A\in\Rmax^{n\times d}$ and $\bm{y}\in\Rmax^n$ recall the \emph{residual} $R:\Rmax^d\mapsto\mathbb{R}$, with $R(\bm{x})=\|A\otimes\bm{x} -\bm{y}\|_{2}^{2}/2$. For a feasible pattern $P$ define the \emph{local residual} $R_{P}: \Rmax^d \mapsto \mathbb{R}$, by 
\begin{equation}
R_{P}(\bm{x})=\|A_{P}(\bm{x})-\bm{y}\|_{2}^2/2.
\end{equation}
For real $\bm{x}\in\R^d$ we have
$$
R_{P}(\bm{x})=\|L\bm{x}+\bm{y}_{P}-\bm{y}\|_{2}^2/2,
$$
Note that for $\bm{x}\in\Cl\big(X(P)\big)$, we have $R(\bm{x})=R_{P}(\bm{x})$. Hence $R$ is piecewise quadratic.

\begin{example}\label{2negg}

Consider 
$$
A=\left[\begin{array}{cc} 0 & 0 \\ 1 & 0 \\ 0 & 1 \end{array}\right], \quad \bm{y}=\left[\begin{array}{cc} 0  \\ 0.5  \\ 0 \end{array}\right], \quad \bm{y}'=\left[\begin{array}{cc} 0  \\ 1.5  \\ 2 \end{array}\right].
$$
Since all entries in $A$, $\bm{y}$ and $\bm{y}'$ are finite both $A,\bm{y}$ and $A,\bm{y}'$ are finite forms. There are seven feasible patterns, their domains are displayed in Figure~\ref{Pdomains}.

\begin{figure}
\begin{center}
\begin{tikzpicture}[scale=3]
    \draw [<->,thick] (0,2) node (yaxis) [above] {$x_2$}
        |- (2,0) node (xaxis) [right] {$x_1$};
    
    \draw (1.5,0) node [rotate=45] {$X\big(P(1)=(\{1\},\{1\},\{1\})\big)$};
    
    \draw (0.5-0.125,-0.5-0.125)--(2,1) node [right] {$X\big(P(2)=(\{1\},\{1\},\{1,2\})\big)$};
    
    \draw (1,0.5) node [rotate=45] {$X\big(P(3)=(\{1\},\{1\},\{2\})\big)$};
    
    \draw (-0.125,-0.125)--(1.5,1.5) node [right]{$X\big(P(4)=(\{1,2\},\{1\},\{2\})\big)$};
    
    \draw (0.5,1) node [rotate=45] {$X\big(P(5)=(\{2\},\{1\},\{2\})\big)$};
    
    \draw (-0.5-0.125,0.5-0.125)--(1,2) node [right] {$X\big(P(6)=(\{2\},\{1,2\},\{2\})\big)$};
    
     \draw (0,1.5) node [rotate=45] {$X\big(P(7)=(\{2\},\{2\},\{2\})\big)$};
    
    
  \end{tikzpicture}
  \end{center}
\caption{Domains of feasible patterns for the matrix $A$ of Example~\ref{2negg}.}\label{Pdomains}
\end{figure}
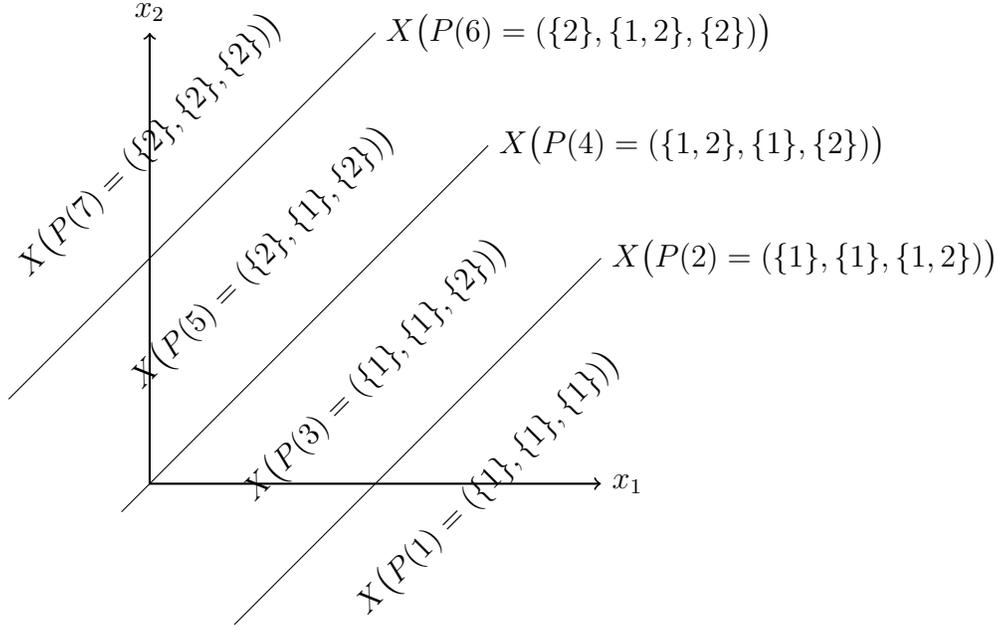

For $P=P(3)=\big(\{1\},\{1\},\{2\}\big)$, the feasibility matrix is given by
$$
F_{P}=\left[\begin{array}{cc} 0 & 0 \\ -1 & 0\end{array}\right].
$$
From Theorem~\ref{feasibilitythm}, since $\lambda(F_{P})=0$, we have that $P$ is admissible and that its domain is given by $X\big(P\big)=\relint\big(\col(F_{P}^{\star})\big)$. In this case $F_{P}^{\star}=F_{P}$ and 
$$
\col(F_{P}^{\star})=\{\bm{x}\in\Rmax^2~:~ x_{2}+1\geq x_{1}\geq x_{2} \}, \quad \relint\big(\col(F_{P}^{\star})\big)=\{\bm{x}\in\Rmax^2~:~ x_{2}+1> x_{1}> x_{2}\}.
$$
The boundary of $X(P)$ is given by $X\big(P(2)\big)\cup X\big(P(4)\big)$. These patterns are both feasible and
\begin{align*}
F_{P(2)}=\left[\begin{array}{cc} 0 & 1 \\ -1 & 0\end{array}\right], \quad &X\big(P(2)\big) = \relint\big(\col(F_{P(2)}^{\star})\big)=\col(F_{P(2)}^{\ast})=\{\bm{x}\in\Rmax^2~:~ x_{1}= x_{2}+1\}, \\ F_{P(4)}=\left[\begin{array}{cc} 0 & 0 \\ 0 & 0\end{array}\right], \quad &X\big(P(4)\big) = \relint\big(\col(F_{P(4)}^{\star})\big)=\col(F_{P(4)}^{\ast})=\{\bm{x}\in\Rmax^2~:~ x_{1}= x_{2}\}.
\end{align*}

\medskip

For $P=\big(\{1,2\},\{2\},\{1\}\big)$, we have
$$
F_{P}=\left[\begin{array}{cc} 0 & 1 \\ 1 & 0\end{array}\right].
$$
Since $\lambda(F_{P})=1$, we have that $P$ is not admissible. 

\medskip

Note that $\bm{z}\in\col(A)$, if and only if $\bm{z}\otimes \alpha=[z_{1}+\alpha,z_{2}+\alpha,z_{3}+\alpha]^{\top}\in\col(A)$, for all $\alpha\in\Rmax$. Therefore $\col(A)$ has translational symmetry in the $[1,1,1]^{\top}$ direction. Similarly for all of the pattern images and extended images. We can therefore study these objects by examining their image under the orthogonal projection $\Pi:\Rmax^3\mapsto\{\bm{z}\in\Rmax^3 ~:~[1,1,1]\bm{z}=0\}$. This is the same idea as in the tropical projected space $\mathbb{TP}^2$, which is usually taken to be a projection onto $\{\bm{y}\in\Rmax^3 ~:~[1,0,0]\bm{y}=0\}$. However the orthogonal projection that we use here is  more convenient for analyzing the $2$-norm regression problem. The projected pattern images,  extended images and a sample of normal projections are displayed in Figure~\ref{Pimages}.

\medskip

We return our attention to the pattern $P=P(3)=\big(\{1\},\{1\},\{2\}\big)$ and set
$$
\bm{x}_{P}=\overline{F_{P}^{\star}}=\overline{F_{P}}=\left[\begin{array}{c} 0 \\ -0.5 \end{array}\right].
$$
We have that $\bowtie_{P}$ is the identity relation with equivalence classes $\{1\}$ and $\{2\}$, which we label as the 1st and 2nd classes respectively. So that $c(1)=1$ and $c(2)=2$ and $C$ is the $2\times 2$ identity matrix. The extended domain is therefore given by
$$
\Cl\Big(\uc{A}\big(X(P)\cap \R^2 \big)\Big)=\Cl\Big(\{C\bm{h}+\bm{x}_{P}~:~ \bm{h}\in \R^2\}\Big)=\Cl(\R^2)=\Rmax^2.
$$
The subpattern of $P$ is given by $p=P$, since $P$ already has exactly one element in each subset. Written as a vector we have $p=[1,1,2]$. The matrix $L$ and the vector $\bm{y}_{P}$ are therefore given by
$$
L=\left[\begin{array}{cc} 1 & 0 \\ 1 & 0 \\ 0 & 1 \end{array}\right], \quad \bm{y}_{P}=\left[\begin{array}{cc} 0 \\ 1  \\  1 \end{array}\right].
$$
The local map is given by $A_{P}(\bm{x})=L\bm{x}+\bm{y}_{P}$. The extended image is given by
\begin{align*}
\Cl\Big(\uc{A}\big(Y(P)\cap \R^3 \big)\Big) &=\Cl\Big(\{ LC\bm{h}+L\bm{x}_{P}+\bm{y}_{P}~:~\bm{h}\in\R^2 \}\Big) \\ &=\Cl\big(\{ \bm{z}\in\R^3~:~ z_{2}=z_{1}+1 \}\big) \\ &=\{ \bm{z}\in\Rmax^3~:~ z_{2}=z_{1}+1 \}.
\end{align*}
and the image is given by
$$
Y(P)=\{\bm{z}\in\Rmax^3~:~ z_{2}=z_{1}+1,~ z_{3}>z_{1}>z_{3}-1 \}.
$$
The boundary of the image is given by $Y\big(P(2)\big)\cup Y\big(P(4)\big)$, where
$$
Y\big(P(2)\big)=\{\bm{z}\in\Rmax^3~:~ z_{2}=z_{1}+1,~ z_{3}=z_{1} \}, \quad Y\big(P(4)\big)=\{\bm{z}\in\Rmax^3~:~ z_{2}=z_{1}+1,~ z_{3}=z_{1}+1 \}.
$$
The normal projection is given by
$$
\Phi(P,\bm{y})=\left[\begin{array}{c} \overline{\{(\bm{y}-L\bm{x}_{P}-\bm{y}_{P})_{j}~:~ j\hat{\bowtie}_{P} 1\}} \\ \overline{\{(\bm{y}-L\bm{x}_{P}-\bm{y}_{P})_{j}~:~ j\hat{\bowtie}_{P} 2\}} \\ \overline{\{(\bm{y}-L\bm{x}_{P}-\bm{y}_{P})_{j}~:~ j\hat{\bowtie}_{P} 3\}}  \end{array}\right]+L\bm{x}_{P}+\bm{y}_{P},
$$
where
$$
\bm{y}-L\bm{x}_{P}-\bm{y}_{P}=\left[\begin{array}{c} 0 \\ -0.5 \\ -0.5 \end{array}\right], \quad L\bm{x}_{P}+\bm{y}_{P}=\left[\begin{array}{c} 0 \\ 1 \\ 0.5 \end{array}\right],
$$
and since the equivalence relation $\hat{\bowtie}_{P}$ has equivalence classes $\{1,2\}$ and $\{3\}$, we have 
$$
\Phi(P,\bm{y})=[-0.25,0.75,0]^{\top}.
$$ 
We have $\support(P)=\{1,2\}$, so that $\ker(LC)=\{0\}$ and therefore 
$$
A_{P}^{-1}\big(\Phi(P,\bm{y})\big)=\Cl\big(C\bm{h}^{\ast}+\bm{x}_{P}\big),
$$
where
$$
\bm{h}^{\ast}=\left[\begin{array}{c} \sqrt{1} \times \overline{\{(\bm{y}-L\bm{x}_{P}-\bm{y}_{P})_{i}~:~ c\big(p(i)\big)=1\}} \\ \sqrt{1} \times \overline{\{(\bm{y}-L\bm{x}_{P}-\bm{y}_{P})_{i}~:~ c\big(p(i)\big)=2\}} \end{array}\right]=\left[\begin{array}{c} -0.25 \\ -0.5 \end{array}\right],
$$ 
so that $A_{P}^{-1}\big(\Phi(P,\bm{y})\big)=\{[-0.25,-1]^{\top}\}$. Since $A_{P}^{-1}\big(\Phi(P,\bm{y})\big)$ contains a single vector we have 
$$
\Psi(P,\bm{y},\bm{x})=\arg\min\{\|\bm{x}-\bm{x}'\|_{2}~:~\bm{x}'\in A^{-1}_{P}\big(\Phi(P,\bm{y})\big)\}=[-0.25,-1]^{\top},
$$
for all $\bm{x}\in\Rmax^2$. Therefore $\Psi(P,\bm{y},\underline{-\infty})=[-0.25,-1]^{\top}$. Checking
$$
F_{P}\otimes \Psi(P,\bm{y},\underline{-\infty})=\left[\begin{array}{cc} 0 & 0 \\ -1 & 0\end{array}\right]\otimes\left[\begin{array}{cc} -0.25 \\ -1 \end{array}\right]=\left[\begin{array}{cc} -0.25 \\ -1 \end{array}\right]= \Psi(P,\bm{y},\underline{-\infty}),
$$
we verify by Theorem~\ref{admissthm} that $\Phi(P,\bm{y})$ is admissible. 

For the alternative target vector $\bm{y}'$ we have  $\Phi(P,\bm{y}')=[0.25,1.25,2]^{\top}$,
$$
A_{P}^{-1}\big(\Phi(P,\bm{y}')\big)=C\bm{h}^{\ast}=[0.25,1]^{\top}
$$
and 
$$
F_{P}\otimes \Psi(P,\bm{y}',\underline{-\infty})=\left[\begin{array}{cc} 0 & 0 \\ -1 & 0\end{array}\right]\otimes\left[\begin{array}{cc} 0.25 \\ 1 \end{array}\right]=\left[\begin{array}{cc} 1 \\ 1 \end{array}\right]> \Psi(P,\bm{y}',\underline{-\infty}),
$$
which verifies by Theorem~\ref{admissthm} that $\Phi(P,\bm{y}')$ is not admissible. See Figure~\ref{Pimages}. Note how the normal projection of $\bm{y}$ lies inside the column space of $A$ but that for the alternative target $\bm{y}'$ it does not. 

\begin{figure}
\begin{center}
\begin{tikzpicture}[scale=4]
    \draw [<->,thick] (0,1) node (yaxis) [above] {$(-x_2+x_{3})/\sqrt{2}$}
        |- (1,0) node (xaxis) [below] {$(-2 x_1+x_{2}+x_{3})/\sqrt{6}$};

    \draw[thick,fill] (-0.7071,0.4082) circle [radius=0.02cm] node [below right] {$Y_{1/2}$};
   
    \draw (-0.7071,0.4082)--(0,0.8165)[line width=2];

	\draw ( -0.3535, 0.6123)--( -0.3535, 0.6123) node [below right] {$Y_3$};
   
    \draw[thick,fill] (0,0.8165) circle [radius=0.02cm] node [ right=0.4cm ] {$Y_{4}$};   
    
    \draw ( 0.3535, 0.6123)--( 0.3535, 0.6123) node [below left] {$Y_5$};
    
  \draw[thick,fill] (0.7071,0.4082) circle [radius=0.02cm] node [below left] {$Y_{6/7}$};
    
    \draw (-0.7071,0.4082)--(0,0.8165)[line width=2];

    \draw (0,0.8165)--(0.7071,0.4082) [line width=2];
   
    \draw ( -1.4142 ,  -0.0001)--(0.7954    , 1.2759) [line width=0.5]; 
   
   \draw ( -1.0606   , 0.2041)--( -1.0606   , 0.2041) node[above left ] {$\uc{A}(Y_3)$};

     \draw ( 1.4142 ,  -0.0001)--(-0.3535   , 1.0207) [line width=0.5];

   \draw ( 1.0606   , 0.2041)--( 1.0606   , 0.2041) node[above right ] {$\uc{A}(Y_5)$};

  \draw[thick,fill] ( -0.3536, 0.2041) circle [radius=0.02cm] node [below] {$\bm{y}$};
  
    \draw( -0.3536, 0.2041)-- ( -0.5303, 0.5103) [->,thick,dashed]node [above left] {$\Phi\big(P(3),\bm{y}\big)$};

  \draw[thick,fill] ( +0.3536, 1.4289)circle [radius=0.02cm] node [above] {$\bm{y}'$};
  
      \draw( +0.3536, 1.4289) -- ( 0.5303, 1.1227) [->,thick,dashed]node [below right] {$\Phi\big(P(3),\bm{y}'\big)$};


  \end{tikzpicture}
  \end{center}
\caption{For the problem of example  Example~\ref{2negg}. Projected pattern images and extended images for the matrix $A$. $Y_{1/2}=Y\big(P(1)\big)=\uc{A}\Big(Y\big(P(1)\big)\Big)=Y\big(P(2)\big)=\uc{A}\Big(Y\big(P(2)\big)\Big)$, $Y_3=Y\big(P(3)\big)$, $Y_4=Y\big(P(4)\big)=\uc{A}\Big(Y\big(P(4)\big)\Big)$, 
$Y_5=Y\big(P(5)\big)$ and $Y_{6/7}=Y\big(P(6)\big)=\uc{A}\Big(Y\big(P(6)\big)\Big)=Y\big(P(7)\big)=\uc{A}\Big(Y\big(P(7)\big)\Big)$. Target vectors $\bm{y}$ and $\bm{y}'$ with normal projections onto $\uc{A}(Y_{3})$.
}\label{Pimages}
\end{figure}
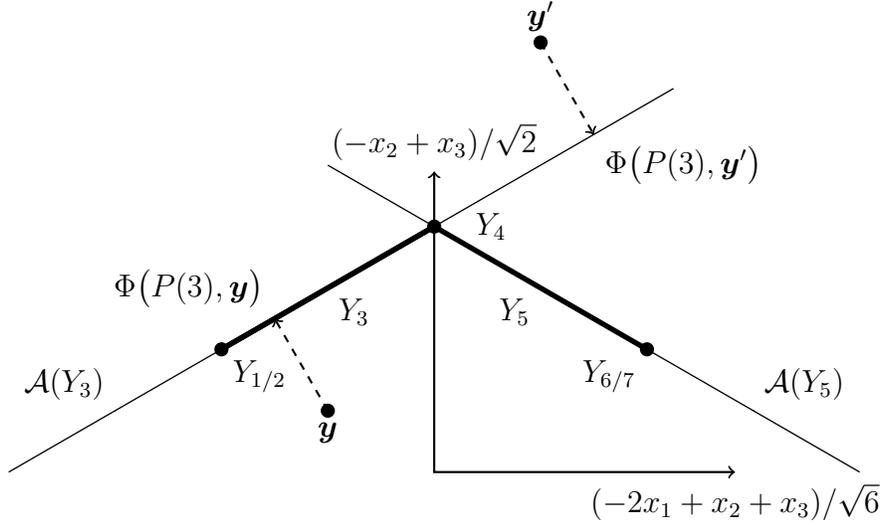

\end{example}

\subsection{NP-hardness of finding descent directions}
For $A\in\Rmax^{n\times d}$, $\bm{y}\in\Rmax^n$ and $\bm{x}\in\Rmax^d$ we say that $\bm{z}\in\R^d$ is an \emph{descent direction} for $\bm{x}$ if there exists $\epsilon>0$ such that
\begin{equation}\label{desvent}
R(\bm{x}+\mu \bm{z})<R(\bm{x}),
\end{equation}
for all $0<\mu\leq \epsilon$.

The following problem is equivalent to determining whether $\bm{x}\in\Rmax^d$ is a local minimum of $R$. In particular $\bm{x}$ is a local minimum if and only if it does not have any descent directions. 
\begin{problem}\label{descent}
Let $A\in\Rmax^{n\times d}$, $\bm{y}\in\Rmax^n$ and $\bm{x}\in\Rmax^d$. Does there exist a descent direction for $\bm{x}$?
\end{problem}

The following problem is known as the set covering problem. 
\begin{problem}\label{setcovering}
Let $F=\big\{F_{i}\subset\{1,\dots,n\}~:~i=1,\dots,m\big\}$ be a family of subsets with $\cup_{i=1}^{m}F_{i}=\{1,\dots,n\}$ and let $1<k<m$. Does there exist a subset $J\subset\{1,\dots,m\}$, such that  $|J|\leq k$ and $\cup_{j\in J}F_{j}=\{1,\dots,n\}$?
\end{problem}

We will now prove that any instance of Problem~\ref{setcovering} can be reduced to an instance of Problem~\ref{descent}. Because the set covering problem is known to be NP-hard this suffices to show that Problem~\ref{descent} is also NP-hard \cite{Karp1972}.

\begin{lemma}\label{pointdowny}
Let $A\in\{ 0,-\infty\}^{n\times d}$ and $\bm{y}\in\R^n$ be a finite form. Then $\bm{z}\in\R^d$ is a descent direction for the vector of zeros $\underline{0}\in\Rmax^d$, if and only if
\begin{equation}\label{dotpro}
\langle A\otimes \bm{z},\bm{y}\rangle>0.
\end{equation}
Moreover, in the case that $\sum_{i=1}^{n}y_{i}=0$, if $\underline{0}\in\Rmax^d$ has a descent direction $\bm{z}\in\R^d$ then is has a descent direction $\bm{z}'\in\{0,1\}^{d}$.
\end{lemma} 
\begin{proof}
Note that since $A\in\{0,-\infty\}^{n\times d}$, we have $A\otimes (\underline{0}+\mu\bm{z})=\mu(A\otimes \bm{z})$, for all $\mu\geq 0$. From \eqref{desvent} we have that $\bm{z}\in\R^d$ is a descent direction for $\underline{0}$ if and only if there exists $\epsilon>0$, such that
$$
R(\underline{0}+\mu \bm{z})=\|\mu(A\otimes \bm{z})-\bm{y}\|_{2}^{2}/2<R(\underline{0})=\|\bm{y}\|_{2}^{2}/2
$$
for all $0<\mu\leq \epsilon$. Which is equivalent to 
$$
\lim_{\mu\rightarrow 0_{+}}\frac{\delta \big(\|\mu(A\otimes \bm{z})-\bm{y}\|_{2}^{2}/2\big)}{\delta \mu}=-\langle A\otimes \bm{z},\bm{y}\rangle<0.
$$

For the second part suppose that $\sum_{i=1}^{n}y_{i}=0$ and that $\bm{z}\in\R^d$ is a descent direction for $\underline{0}$. Note that for any $a,b\in\R$ with $a>0$, we have that $a\bm{z}+b\underline{1}$ is also a descent direction, where $\underline{1}\in\R^d$ is a vector of ones. Furthermore since $\sum_{i=1}^{n}y_{i}=0$, we must have that $\min_{i=1}^d z_{i}<\max_{i=1}^d z_{i}$. So we can assume without loss of generality that $\min_{i=1}^d z_{i}=0$ and $\max_{i=1}^d z_{i}=1$. Now let $\sigma$ be a permutation of $\{1,\dots,d\}$ such that
$$
z_{\sigma(1)}\leq z_{\sigma(2)}\leq \cdots \leq z_{\sigma(d)}.
$$
And define the sequence $b_{1}<b_{2}<\cdots< b_{m}$ such that $b_{1}=1$, $b_{m}=d$ and
$$
z_{\sigma(i)}=z_{\sigma(j)} ~ \Leftrightarrow ~b_{k}\leq i,j <b_{k+1}~\hbox{for some $k$}.
$$
Define the vectors $\big(\bm{h}(k)\in\R^d\big)_{k=1}^{m}$, by
$$
h(k)_{\sigma(j)}=\left\{\begin{array}{cc} 1 & \hbox{if $j\geq b_{k}$} \\ 0 & \hbox{otherwise.}\end{array}\right.
$$
for $j=1,\dots,d$. Note that for $r,j=1,\dots,d$ and $k=1,\dots,m$, we have
\begin{equation}\label{onetwo}
z_{r}>z_{j}~\Rightarrow h(k)_{r}\geq h(k)_{j},  \quad \hbox{and} \quad z_{r}=z_{j}~\Rightarrow h(k)_{r}= h(k)_{j}.
\end{equation}
Let $P=\pattern(\bm{z})$ and let $Q(k)=\pattern\big(\bm{h}(k)\big)$ for $k=1,\dots,m$. Now suppose that $j\in P_{i}$, then 
$$
a_{ir}=-\infty ~\forall ~r~:~ z_{r}>z_{j},  \quad \hbox{and} \quad \exists~s~:~z_{s}=z_{j},~a_{is}=0.
$$
So from \eqref{onetwo} we have that
$$
a_{ir}=-\infty ~\forall ~r~:~ h(k)_{r}>h(k)_{j},  \quad \hbox{and} \quad \exists~s~:~h(k)_{s}=h(k)_{j},~a_{is}=0,
$$
and that $j\in Q(k)_{i}$, for all $k=1,\dots,m$. Therefore $Q(k)\preceq P$ and $\bm{h}(k)\in \Cl\big(X(P)\big)$, so that
\begin{equation}\label{eeeee}
A\otimes \bm{h}(k)=A_{P}\big(\bm{h}(k)\big),
\end{equation}
for all $k=1,\dots,m$. The the local map in \eqref{eeeee} is given by 
$$
A_{P}(\bm{x})=L\bm{x}+\bm{y}_{P}, 
$$
where the vector $\bm{y}_{P}$ is defined by $(\bm{y}_{P})_{i}=a_{ip(i)}$, where $p$ is the subpattern of $P$. Note that since $A,y$ are a finite form and $P$ is a feasible pattern and since $A$ has only zero and minus infinity entries, we must have $\bm{y}_{P}=\underline{0}\in\R^n$ and therefore that $A_{P}$ is a classically linear map. 

Finally note that 
\begin{equation}\label{toto}
\bm{z}=\sum_{k=2}^{m}\bm{h}(k)\alpha_{k},
\end{equation}
where the coefficients
$$
\alpha_{k}=\left(z_{\sigma\big(b(k)\big)}-z_{\sigma\big(b(k-1)\big)}\right),
$$
are strictly positive for $k=2,\dots,m.$ From \eqref{eeeee}, \eqref{toto} and the fact that $A_{P}$ is linear, we obtain
\begin{equation}\label{meme}
 \langle A\otimes \bm{z},\bm{y}\rangle=\sum_{k=2}^{m}\alpha_{k}\langle A\otimes \bm{h}(k),\bm{y}\rangle.
\end{equation}
Therefore, using the result of the first part of the Lemma, if $z\in\R^d$ is a descent direction  for $\underline{0}\in\R^d$ then $\langle A\otimes \bm{z},\bm{y}\rangle>0$ and from \eqref{meme} there exists $2\leq k\leq m$ such that  $\langle A\otimes \bm{h}(k),\bm{y}\rangle>0$, equivalently, such that $\bm{h}(k)$ is a descent direction for $\underline{0}\in\R^d$ with all entries in $\{0,1\}$.

\end{proof}

\begin{theorem}\label{setsetset}
Let $F=\{F_{i}~:~i=1,\dots,m\}$ be a family of subsets of $\{1,\dots,n\}$ with $\cup_{i=1}^{m} F_{i}=\{1,\dots,n\}$ and let $1< k< m$. There exists a subset $J \subset\{1,\dots,m\}$, such that $|J|\leq k$ and $\cup_{j\in J} F_{j}=\{1,\dots,n\}$, if and only if $\underline{0}\in\Rmax^m$ has a descent direction for the following $N\times m$ max-plus $2$ norm regression problem, where $N=n+m+\frac{m(m-1)}{2}+1.$ Set
$$
A=\left[\begin{array}{c} A_{1} \\ A_{2} \\ A_{3} \\ A_{4} \end{array}\right], \quad \bm{y}=\left[\begin{array}{c} \bm{y}_{1} \\ \bm{y}_{2} \\ \bm{y}_{3} \\ \bm{y}_{4} \end{array}\right],
$$
where 
\begin{enumerate}
\item $A_{1}\in\Rmax^{n\times m}$ is defined by $(A_{1})_{ij}=0$, if $i\in F_{j}$ and $(A_{1})_{ij}=-\infty$, otherwise, for $i=1,\dots,n$, $j=1,\dots,m$. $\bm{y}_{1}\in\Rmax^n$ is defined by $(\bm{y}_{1})_{i}=a$ for all $i=1,\dots,n$. 

\item $A_{2}\in\Rmax^{m\times m}$ is the $m\times m$ max-plus identity matrix defined by $(A_{2})_{ij}=0$, if $i=j$ and $(A_{2})_{ij}=-\infty$, otherwise, for $i,j=1,\dots,m$. $\bm{y}_{2}\in\Rmax^m$ is defined by $(\bm{y}_{2})_{i}=b$ for all $i=1,\dots,m$.

\item $A_{3}\in\Rmax^{{m(m-1)}/{2}\times m}$ is defined by $(A_{3})_{ij}=0$, for $j\in p(i)$ and $(A_{3})_{ij}=-\infty$, otherwise, for $i=1,\dots,{m(m-1)}/{2}$, $j=1,\dots,m$, where $p(1),p(2),\dots,p({m(m-1}/{2})$ is a list of all the unordered pairs of elements of $\{1,\dots,m\}$. $\bm{y}_{3}\in\Rmax^{{m(m-1)}/{2}}$ is defined by $(\bm{y}_{3})_{i}=c$, for all $i=1,\dots,{m(m-1)}/{2}$.

\item $A_{4}\in\Rmax^{1\times m}$ is defined by $(A_{4})_{1j}=0$ for $j=1,\dots,m$. $\bm{y}_{4}\in\Rmax$ is given by 
$$
\bm{y}_{4}=-na-mb-\frac{m(m-1)}{2}c.
$$ 
\end{enumerate}
Where the coefficients are given by
$$
a=m(k+1), \quad b=m-k-3/2, \quad c=-2.
$$
\end{theorem}
\begin{proof} First note that 
$$
\sum_{i=1}^{N}y_{i}=0,
$$
and that $A,\bm{y}$ are a finite form so the results of Lemma~\ref{pointdowny} apply. Now suppose that $\bm{z}\in\{0,1\}^m$ and let $J=\{j~:~z_{j}=1\}$, $x=|J|$ and $u=|\cup_{j\in J} F_{j}|$. Then we have
\begin{align*}
\langle A\otimes \bm{z},\bm{y}\rangle &=ua+xb+\big(x(m-x)+\frac{x(x-1)}{2}\big)c-na-mb-\frac{m(m-1)}{2}c \\ &=(x-m)\big(x-(k+1/2)\big)+a(u-n),
\end{align*}
unless $x=0$, in which case we have $\langle A\otimes \bm{z},\bm{y}\rangle=0$. Therefore since $0\leq x\leq m$ and $0\leq u\leq n$, we have
\begin{equation}\label{ifff}
\langle A\otimes \bm{z},\bm{y}\rangle>0 ~ \Leftrightarrow ~ u=n, ~x\leq k.
\end{equation}

\smallskip

Now suppose that the set covering problem has a solution $J\subset\{1,\dots,m\}$. Then $|J|\leq k$ and $|\cup_{j\in J} F_{j}|=n$. Define $\bm{z}\in\{0,1\}^{m}$, by $z_{j}=1$, if and only if $j\in J$, then from \eqref{ifff} we have $\langle A\otimes \bm{z},\bm{y}\rangle>0$, which from the first part of Lemma~\ref{pointdowny} is equivalent to $\bm{z}$ being a descent direction for $\underline{0}$. Conversely suppose that $\bm{z}\in\R^m$ is a descent direction for $\underline{0}$, then from the second part of Lemma~\ref{pointdowny} we have that there exists a descent direction $\bm{z}'\in\{0,1\}^{m}$ and from the first part of Lemma~\ref{pointdowny} we  have that $\langle A\otimes \bm{z}',\bm{y}\rangle>0$. Now let $J=\{j~:~z_{j}'=1\}$, then from \eqref{ifff} we must have that $|J|\leq k$ and $|\cup_{j\in J} F_{j}|=n$, i.e. that $J$ is a solution to the set covering problem.

\end{proof}

\begin{corollary}
Problem~\ref{descent} is NP-hard.
\end{corollary}
\begin{proof}
Theorem~\ref{setsetset} shows that any instance of the set-covering problem can be reduced to an instance of Problem~\ref{descent} with size polynomial in that of the original problem. From \cite{Karp1972} we have that the set-covering problem is NP-hard and therefore Problem~\ref{descent} is also NP-hard. 
\end{proof}

\subsection{Brute force method}

An exhaustive approach to solving the max-plus $2$-norm regression problem exactly is to search through all of the feasible patterns, computing the normal projections and checking their admissibility for each one in turn. We  then select the closest admissible normal projection for our solution.

To search efficiently through the set of all feasible patterns we consider the tree $T$, with vertices $V_{0},\dots, V_{n}$ at depths $0,1,\dots,n$ respectively. The depth $k$ vertices $V_{k}$ are ordered $k$-tuples of the form $v=(P_{1},\dots,P_{k})$, with $P_{i}\subset\{1,\dots,d\}$, for $i=1,\dots,k$.  The depth $k$ vertices can be interpreted as patterns of support for the max-plus regression problem formed from the first $k$ rows of $A$ and $\bm{y}$.  A vertex $v\in V_{k}$ is parent to  $v'\in V_{k+1}$ if and only if $v'=(v,P_{k+1})$ for some $P_{k+1}\subset\{1,\dots,d\}$. 

In analogy to the feasibility matrix for a pattern \eqref{feasibilitymatrixdef}, define the feasibility matrix for a vertex $v\in V_{k}$, by $F_{v}\in\Rmax^{d\times d}$, with 
\begin{equation}
f_{jk}=\left\{\begin{array}{cc} 0, & \hbox{for $j=k$,} \\ \max\big\{-\infty,\max \{a_{ik}-a_{ij}~:~ j\in P_{i}, ~ i=1,\dots,k\}\big\},   & \hbox{otherwise.}\end{array}\right.
\end{equation}
We say that $v\in V_{k}$ is feasible if $\lambda(F_{v}^{\star})=0$. Note that the set of feasible leaf vertices is identical to the set of feasible patterns of support. It is easy to show that if $v\in V_{k}$ is feasible then all of its ancestors are feasible and at least one of its children is feasible. Also if $v\in V_{k}$ is not feasible then none of its children are feasible.

Next we define an order $\trianglelefteq_{L}$ on the vertices of $T$ by taking an arbitrary ordering $\trianglelefteq$ on the subsets of $\{1,\dots,d\}$ and extending this order lexicographically to $T$. We start at the vertex $v_{0}=()$ and proceed to search through $T$, in order of $\trianglelefteq_{L}$. At each vertex $v\in V_{k}$, we check for feasibility by computing $\lambda(F_{v}^{\star})$, with worst case cost $\Theta(d^3)$. If $\lambda(F_{v}^{\star})>0$, then $v$ is non-feasible and we skip all of its decedents. Whenever we reach a feasible leaf vertex $P\in V_{n}$, we compute $\Phi(P,\bm{y})$ and $\Psi(P,\bm{y},\underline{-\infty})$ with cost $\Theta(n)$. We check for admissibility by computing $F_{P}\otimes \Psi(P,\bm{y},\underline{-\infty})$ with cost $\Theta(d^2)$. If $F_{P}^{}\otimes \Psi(P,\bm{y},\underline{-\infty})=\Psi(P,\bm{y},\underline{-\infty})$ then we compute $\|\bm{y}-\Phi(P,\bm{y})\|_{2}$ and if this residual is the best that we have seen so far, then we save $\Psi(P,\bm{y},\underline{-\infty})$ as the interim optimal solution. The algorithm terminates when we have either checked or skipped all vertices in $T$. See Algorithm~\ref{brute}.

\begin{algorithm}
\caption{ \label{brute}
Returns an optimal solution to Problem~\ref{mppnreg} with $p=2$.}

\medskip

\begin{algorithmic}[1]

\State set $v=()$
\State set $r_{\min}=\infty$
\While{not all of $T$ explored or skipped}
\If{$\lambda(F_{v}^{\star})=0$ $\Leftrightarrow$ $v$ is feasible}
\If {$v=P$ is a leaf vertex}
\State compute $\Phi(P,\bm{y})$ and $\Psi(P,\bm{y},\underline{-\infty})$
\If {$F_{P}^{}\otimes \Psi(P,\bm{y},\underline{-\infty})=\Psi(P,\bm{y},\underline{-\infty})$  $\Leftrightarrow$ $\Phi(P,\bm{y})$ is admissible}
\If {$\|\Phi(P,\bm{y})-\bm{y}\|_{2}<r_{\min}$}
\State update $r_{\min}\mapsfrom \|\Phi(P,\bm{y})-\bm{y}\|_{2}$ and $x\bm{x}\mapsfrom \Psi(P,\bm{y},\underline{-\infty})$
\EndIf
\EndIf
\EndIf
\Else
\State  skip decedents of $v$
\EndIf
\State go to next $v$
\EndWhile
\State return solution $\bm{x}$

\end{algorithmic}

\end{algorithm}

Applied to an $n\times d$ problem, Algorithm~\ref{brute} must check the feasibility of $c_{v}$ vertices and compute normal projections and check admissibility for $c_{l}$ leaf vertices. From \eqref{numer} we have 
$$
c_{l}\leq\sum_{k=1}^{n} \frac{(n+d-k-1)!}{(n-k)!\cdot (d-k)! \cdot (k-1)!},
$$
with equality in the generic case. The total cost of Algorithm~\ref{brute} is $\Theta(c_{v}d^3+c_{l}d^2)$. We do not currently have any convenient expression for $c_{v}$ but the contribution to the cost from the leaf vertices alone is greater than polynomial. Although such exhaustive search algorithms will never be suitable for applying to very large problems, they can still be extremely valuable for use on smaller data sets and as a way to benchmark the performance of faster approximate algorithms.

%

\subsection{Newton's method}

The results of the previous section and Theorem~\ref{setsetset} suggest that computing exact local minima for the max-plus $p=2$ regression problem might not be computationally feasible for larger problems. Instead we propose using the following technique, which consists of Newton's method with an undershooting parameter.  

Recall that the squared residual $R:\Rmax^d\mapsto\mathbb{R}_{+}$, given by $R(\bm{x})=\|A\otimes\bm{x} -\bm{y}\|_{2}^{2}/2$ is picewise quadratic and that  for $\bm{x}\in\Cl\big(X(P)\big)$, we have $R(\bm{x})=R_{P}(\bm{x})$.
Newton's method minimizes a function by iteratively mapping to the minimum of a local quadratic approximation to that function. In the case of the squared residual $R$ this means iteratively mapping to the minimum of the locally quadratic piece. There are several options when implementing Newton's method, for example when $\bm{x}$ is contained in the closure of more than one domain, which pattern do we choose? Also, how do we choose between non-unique minima? The method we set out below is chosen primarily for its simplicity. 
\medskip

For $\bm{x}\in\Rmax^d$, define the \emph{subpattern} $p(\bm{x})$ to be the subpattern of $\pattern(\bm{x})$ 
\begin{equation}
p(\bm{x})_{i}=\min\big(\pattern(\bm{x})_{i}\big).
\end{equation}
Then $p(\bm{x})\preceq \pattern(\bm{x})$ and $\bm{x}\in\Cl\Big(X\big(p(\bm{x})\big)\Big)$.
Define the \emph{Newton update map} $\mathcal{N}:\Rmax^d\mapsto \Rmax^d$, by
\begin{equation}\label{newtonmappp}
\mathcal{N}(\bm{x})=\Psi\big(p(\bm{x}),\bm{y},\bm{x}\big).
\end{equation}
The map \eqref{newtonmappp} is set to always chooses a pattern whose domain is of the maximum possible dimension. In the case where the minima is non-unique, it returns the one that is closest to the current point.

A difficulty for Newton's method is that the non-differentiability of $R$ means that the iteration needn't converge to a local minima and can instead get caught in a periodic orbit. This makes choosing a stopping condition difficult. We use the rule that if the residual has not decreased in some fixed number of steps then we terminate the algorithm and return the best solution from the iterations orbit. We also include a shooting parameter $\mu\in(0,1)$ and make the update $\bm{x}\mapsfrom (1-\mu)\bm{x}+\mu \mathcal{N}(\bm{x})$. Choosing $\mu<1$ causes the method to undershoot and so avoid being caught in the periodic orbits mentioned previously. If Algorithm~\ref{newton} iterates $k$ times then it has cost $\Theta(knd)$. In the Numerical examples that follow we randomly sample ten different initial conditions then apply Algorithm~\ref{newton} once with $\mu=1$ then once more with $\mu=0.05$, each time using $t=5$, then pick the best approximate solution. Optimizing the choice of parameters and random starting conditions is an important topic for future research. 

\begin{algorithm}
\caption{ \label{newton} {\bf (Newton's method)} 
Given an initial guess $\bm{x}$, returns an approximate solution to Problem~\ref{mppnreg} with $p=2$. Parameters are $t\in\mathbb{N}$ the number of iterations for stopping condition and $\mu\in(0,1)$ the undershooting parameter, which may be allowed to vary during the computation.}

\medskip

\begin{algorithmic}[1]

\State set $r_{\min}=\infty$
\While{not terminated}
\State update $\bm{x}\mapsfrom (1-\mu)\bm{x}+\mu \mathcal{N}(\bm{x})$
\If {$R(\bm{x})<r_{\min}$} $r_{\min}=R(\bm{x})$, $\hat{\bm{x}}=\bm{x}$, \EndIf
\If {$r_{\min}$ not decreased for $t$ iterations} terminate
\EndIf
\EndWhile
\State return approximate solution $\hat{\bm{x}}$

\end{algorithmic}

\end{algorithm}

Applied to an $n\times d$ problem Algorithm~\ref{newton} has cost $\Theta(nd)$ per iteration.

\section{System identification}\label{tss}

Consider the $d$-dimensional stochastic max-plus linear dynamical system 
\begin{equation}\label{dysys}
\bm{x}(n+1)=M\otimes \bm{x}(n)+\zeta(n),
\end{equation}
where $M\in\Rmax^{d\times d}$ and $\zeta(0),\zeta(1),\dots\in\R^d$ are i.i.d Gaussians with mean zero and covariance matrix $\sigma^2 I$. Suppose that we do not know $M$, but that we have observed an orbit $\bm{x}(0),\bm{x}(1),\dots,\bm{x}(N)$ and want to estimate $M$ from this data. The maximum likelihood estimate for this inference problem is given by
\begin{equation}\label{invdsres}
\max_{A\in\Rmax^{d\times d}}\mathbb{P}\{\bm{x}(0),\bm{x}(1)\dots,\bm{x}(N)~|~ \bm{x}(k+1)=A\otimes \bm{x}(k)+\zeta(k), ~ k=0,1,\dots,N-1 \}.
\end{equation}
This problem can be expressed as $d$ independent regression problems as follows. Expanding \eqref{invdsres} yields
\begin{align}
\mathbb{P}\{\bm{x}(0),\dots,\bm{x}(N) ~|~ A\} &=\prod_{n=0}^{N-1}\mathbb{P}\{\zeta(n)=\bm{x}(n+1)-A\otimes \bm{x}(n)\} \\  &=\prod_{n=0}^{N-1}\prod_{k=1}^{d}\frac{1}{\sqrt{2\pi\sigma^2}}\exp \left({\frac{-\Big(\bm{x}(n+1)-A\otimes \bm{x}(n)\Big)_{k}^2}{2\sigma^2}}\right).
\end{align}
The negative log likelihood is therefore given by
\begin{equation}\label{llhood}
-\log\big(\mathbb{P}\{\bm{x}(0),\dots,\bm{x}(N) ~|~ A\}\big) =\frac{Nd}{2}\log(2\pi\sigma^2)+\frac{1}{2\sigma^2}\|A\otimes X(:,1:N)-X(:,2:N+1)\|_{F}^{2},
\end{equation}
where $X\in\Rmax^{d\times (N+1)}$ is the matrix whose columns are the time series observations $\bm{x}(0),\dots,\bm{x}(N)$ and where we use the Matlab style notation $X(\mathcal{I},\mathcal{J})$ to indicate the submatrix of formed from the intersection of the $\mathcal{I}$ rows and $\mathcal{J}$ columns of $X$ and use the symbol $:$ alone to denote the full range of row/cols. Next note that
\begin{equation}\label{ffffff}
\|A\otimes X(:,1:N)-X(:,2:N+1)\|_{F}^{2}=\sum_{k=1}^{d}\|A(k,:)\otimes X(:,1:N)-X(k,2:N+1)\|_2^2.
\end{equation}
Maximizing \eqref{invdsres} is therefore equivalent to minimizing each of the terms summed over in \eqref{ffffff}. The $k$th of these terms measures our model's ability to predict the value of the $k$th variable at the next time step. To minimize this error we choose the $k$th row of $A$ by
\begin{equation}\label{rowA}
A(k,:)=\arg\min_{\bm{x}\in\Rmax^{1\times d}}\|\bm{x}\otimes X(:,1:N) -X(k,2:N+1)\|_{2},
\end{equation}
which requires us to solve an $n\times d$ max-plus $2$-norm regression problem. We can therefore solve \eqref{invdsres} by solving $d$ such regression problems. 

\begin{example}\label{invpegg}
Consider the matrix
$$
M=\left[\begin{array}{cccc} 7 &  15 &  10 &  -\infty   \\  14 &  -\infty &  11 &  11  \\  14 &  -\infty &  -\infty &  -\infty   \\  15 &  8 &  7 &  9  \\  \end{array}\right].
$$
From the initial condition $\bm{x}(0)=[0,0,0,0]^{\top}$ we generate two orbits of length $n=200$ by iterating \eqref{dysys}. One with a low noise level, $\sigma=1$ and one with a  high noise level, $\sigma=5$. Next we compute the maximum likelihood estimate for $M$ from the time series, by applying Algorithm~\ref{newton} to each of the row problems \eqref{rowA}, for $k=1,\dots,d$. Our estimates are given by
$$ 
  A(\sigma=1)=\left[\begin{array}{cccc} 2.65 &  14.9 &  10.6 &  10  \\  13.8 &  -30.4 &  -53.2 &  10.9  \\  14 &  7.06 &  8.5 &  5.7  \\  15 &  9.4 &  -50.3 &  8.28  \\  \end{array}\right], \quad A(\sigma=5)=\left[\begin{array}{cccc} 8.24 &  14.1 &  11 &  1.67  \\  13.8 &  -\infty &  9.28 &  11.1  \\  13.8 &  -\infty &  -\infty &  -\infty   \\  14.3 &  9.21 &  7.49 &  7.62  \\  \end{array}\right].
 $$
 Table~\ref{resftab} displays the Frobenius error term \eqref{ffffff} for each of these estimates. Note that both of these estimates fit the data better than the true system matrix $M$, which indicates that Algorithm~\ref{newton} is able to find close to optimal solutions to the regression problem. 
 
Comparing our estimates to $M$, we see that in both cases we have inferred values that are roughly correct for the larger entries in the matrix but that the minus infinities are poorly approximated in both cases and that some of the smaller finite entires are poorly approximated in the low noise case. For each orbit we record the matrix $S\in\mathbb{N}^{4\times 4}$, with
$$
s_{ij}=|\{0\leq n<N-1 ~:~\big(A\otimes x(n)\big)_{i}=a_{ij}+x(n)_{j}\}|,
$$
which records how often variable $j$ attains the maximum in determining variable $i$ at the next time step, for $i,j=1,\dots,4$. Therefore $s_{ij}$ can be  thought of as a measure of how much evidence we have to infer the parameter $a_{ij}$ from the orbit. These matrices are given by
$$
S(\sigma=1)=\left[\begin{array}{cccc} 0 &  201 &  0 &  0  \\  167 &  0 &  5 &  29  \\  201 &  0 &  0 &  0  \\  197 &  0 &  0 &  4  \\  \end{array}\right], \quad S(\sigma=5)=\left[\begin{array}{cccc} 30 &  137 &  34 &  0  \\  108 &  0 &  41 &  52  \\  201 &  0 &  0 &  0  \\  133 &  30 &  12 &  26  \\  \end{array}\right].
$$
Comparing the results it is clear that our inferences are more accurate for entries with more evidence. In the low noise case the evidence is all contained on a small number of entries, as under the nearly deterministic behavior of this regime only a few positions are ever able to attain the maximum. In the high noise case the more random behavior means that more entries are able to attain the maximum and therefore the evidence is more uniformly distributed, except onto the minus infinity entries, which can never attain the maximum.

\end{example}

Inferring the values of entries that do not play a role in the dynamics or only play a very small role is therefore an ill posed problem and consequently we obtain MLE matrices $A$ that do a good job of fitting the data but which are not close to the true system matrix $M$. There are two common strategies for coping with such ill posed inverse problems. The first is to choose a prior distribution for the inferred parameters, then compute a maximum a posteri estimate which minimizes the likelihood times the prior probability. The second approach is to add a regularization penalty to the targeted residual. Regularization is typically used to improved the well-posedness of inverse problems and to promote solutions which are in some way simpler. Typical choices for conventional linear regression problems are the $1$-norm or $2$-norm of the solution. For max-plus linear $2$-norm regression we propose the following regularization penalty, which is chosen to promote solutions $\bm{x}\in\Rmax^d$ with smaller entries and with more entries equal to $-\infty$.

\begin{problem}\label{regl2}
For $A\in\mathbb{R}_{\max}^{n\times d}$, $\bm{y}\in\mathbb{R}_{\max}^{n}$ and $\lambda\geq 0$, we seek
\begin{equation}\label{qqq1}
\min_{\bm{x}\in\mathbb{R}_{\max}^{d}}\Big(\|\big(A\otimes \bm{x}\big)-\bm{y}\|_{2}^2+\lambda\sum_{j=1}^{d}\bm{x}_{j}\Big),
\end{equation}
where we use the convention that in the case when the residual is $+\infty$ and the regularization term is $-\infty$ the total is $+\infty$. When solutions with $-\infty$ totals are possible we order them first by the number of components in $\bm{x}$ equal to $-\infty$ and then by the total in \eqref{qqq1} not including those $-\infty$ terms.  
\end{problem}

Rather than searching for a global optima for Problem~\ref{regl2}, which would contain the largest number of $-\infty$ entries possible, we instead start from a good (ideally optimal) solution to the unregularized problem and search for a nearby solution to Problem~\ref{regl2}. In this way the regularization term provides a kind of downwards pull on the entries with very little evidence but not enough to disrupt the entries for which there is plenty of evidence. 

We solve Problem~\ref{regl2} by solving a sequence of max-plus $2$-norm regression problems by an approach which is inspired by the iteratively reweighed least squares method for solving conventional $1$-norm regularized $2$-norm regression problems \cite{IRLS}.  Let $I\in\Rmax^{d\times d}$ be the max-plus identity matrix with zeros on the diagonal and minus infinities off of the diagonal and consider the residual
\begin{align}
\label{ww1} \left\|\left[\begin{array}{c} A \\ I \end{array}\right]\otimes \bm{x}-\left[\begin{array}{c} \bm{y} \\ \bm{x}'-\lambda/2 \end{array}\right]\right\|_{2}^2&=\|A\otimes \bm{x}-\bm{y}\|_{2}^{2}+\sum_{j=1}^{d}(\lambda/2+\bm{x}-\bm{x}')^2 \\ &=\|A\otimes \bm{x}-\bm{y}\|_{2}^{2}+\lambda\sum_{j=1}^{d}(\bm{x}-\bm{x}')_{j}+\uc{O}\big((\bm{x}-\bm{x}')^2\big).
\end{align}
Therefore, for $\bm{x}$ close to $\bm{x}'\in\R^d$, \eqref{ww1} only differs from \eqref{qqq1} by a constant factor. Algorithm~\ref{irsls} computes a sequence of approximate solutions, each time shifting the  additional target variables in \eqref{ww1} to match the gradient of the residual in \eqref{qqq1}. If a component of the solution appears to be diverging to minus infinity, then we set it to equal this limit.

\begin{algorithm}
\caption{ \label{irsls} {\bf (Iteratively reshifted least squares)} 
Given an initial guess $\bm{x}$, returns an approximate solution to Problem~\ref{regl2}.}

\medskip

\begin{algorithmic}[1]

\While{not converged}
\State apply a max-plus 2-norm regression solver to compute
 $$\bm{x}\mapsfrom \arg\min_{\bm{x}'\in\Rmax^d}\left\|\left[\begin{array}{c} A \\ I \end{array}\right]\otimes \bm{x}'-\left[\begin{array}{c} y \\ \bm{x}(k-1)-\lambda/2 \end{array}\right]\right\|_{2}^2$$
\EndWhile
\State return approximate solution $\bm{x}$

\end{algorithmic}

\end{algorithm}

\begin{example}
Returning to the problem of Example~\ref{invpegg}. We repeat our analysis of the time-series data only this time we include a regularization term when computing each row via \eqref{rowA}. The results of our regularized inference are as follows 
$$
 {A}(\sigma=1,\lambda=10)=\left[\begin{array}{cccc} -\infty &  15 &  -\infty &  -\infty   \\  13.9 &  -\infty &  -\infty &  11  \\  14 &  -\infty &  -\infty &  -\infty   \\  14.9 &  -\infty &  -\infty &  -\infty   \\  \end{array}\right], \quad {A}(\sigma=5,\lambda=10)=\left[\begin{array}{cccc} 8.05 &  14 &  10.5 &  -\infty   \\  13.7 &  -\infty &  9.03 &  10.9  \\  13.8 &  -\infty &  -\infty &  -\infty   \\  14.1 &  8.97 &  7.26 &  7.4  \\  \end{array}\right].
 $$
\end{example}
Note that any entry with little of no evidence is set to minus infinity and that the remaining entries are all fairly accurate approximations of the entries in the true system matrix $M$. Table~\ref{resftab} shows that applying the regularization penalty with $\lambda=10$ only results in a tiny degradation in the solutions fit to the data.

\begin{table}
\centering
\caption{For the inverse problem of Example~\ref{invpegg}. Squared Frobenius norm residual $\|A\otimes X(:,1:n)-X(:,2:n+1)\|_{F}^2$, for low and high noise orbits, with and without regularization penalty. All numeric values given to two decimal places.}\label{resftab}
\begin{tabular}{c|c|c|c}
 & $M$ & $A(\sigma,\lambda=0)$ &  $A(\sigma,\lambda=10)$ \\
\hline
$ \sigma=1$ & 233.78 &  227.41 & 251.86 \\
\hline
$ \sigma=5$ & 5308.58 & 5267.86 & 5275.12
\end{tabular}
\end{table}

\section{Discussion}

In this paper we presented theory and algorithms for max-plus $2$-norm regression and then demonstrated how it could be applied to system identification of max-plus linear systems with Gaussian noise. 


We have shown how the geometry of max-plus linear spaces give rise to non-convex optimization problems. We have also proven that the exponentially many different patterns of support in a max-plus regression problem means that even the simple problem of determining whether a point is a local minimum is NP-hard. 

In spite of these difficulties we did find that Algorithm~\ref{newton} worked well enough in our example problem. However, developing efficient algorithms that are able to provide some better performance guarantees would be very desirable. Theorem~\ref{setsetset} seems to stand somewhat in the way of this goal, but note that the theorem relates to a specific point for a highly structured (i.e. degenerate) problem. So it may still be possible to developing an efficient residual descending algorithm.

As noted in the introduction, most applications of max-plus linear dynamical systems use petri-net models, which result in highly structured iteration matrices and noise processes. Further work is needed to adapt the approach used in Section~\ref{tss} to this setting. Similarly to include a control input as in \cite{1184996}. Framing these more general inverse problems explicitly in terms of linear regression problems might inspire new techniques, possibly by trying to develop further max-plus analogues of classical linear systems theory.  
%

\section*{Acknowledgement} We thank Henning Makholm for answering a question on Math Stack Exchange, which helped in the formulation of Theorem~\ref{setsetset}.

\bibliographystyle{is-abbrv}
\bibliography{trop}

\newpage

\end{document}